\documentclass[11pt]{amsart}

\usepackage{amsthm, amssymb, amsmath, stmaryrd}

\input xy
\xyoption{all}

\newcommand{\quash}[1]{}  %%Anything in \quash is ignored
\newcommand{\isom}{\stackrel{\sim}{\to}}

\newcommand{\const}[1]{\overline{\QQ}_{\ell,#1}}
\newcommand{\twtimes}[1]{\stackrel{#1}{\times}}

\newcommand{\homo}[2]{\mathbf{H}_{#1}({#2})}   % sheaf
\newcommand{\homog}[2]{\textup{H}_{#1}({#2})}  % plain group
\newcommand{\cohog}[2]{\textup{H}^{#1}({#2})}     % plain group
\newcommand{\cohoc}[2]{\textup{H}_{c}^{#1}({#2})}     % compact support
\newcommand{\hBM}[2]{\textup{H}^{\textup{BM}}_{#1}({#2})}  % Borel-Moore

\newcommand{\oleft}[1]{\overleftarrow{#1}}
\newcommand{\oright}[1]{\overrightarrow{#1}}

\DeclareMathOperator{\Hom}{Hom}
\DeclareMathOperator{\End}{End}
\DeclareMathOperator{\Ext}{Ext}
\DeclareMathOperator{\Aut}{Aut}

\DeclareMathOperator{\Spec}{Spec}
\DeclareMathOperator{\Res}{Res}

\DeclareMathOperator{\Gal}{Gal}

\DeclareMathOperator{\Gr}{Gr}
\DeclareMathOperator{\id}{id}
\DeclareMathOperator{\ev}{ev}
\DeclareMathOperator{\Irr}{Irr}
\DeclareMathOperator{\Pic}{Pic}

\DeclareMathOperator{\rk}{rk}
\DeclareMathOperator{\Lie}{Lie}

\DeclareMathOperator{\reg}{reg}
\DeclareMathOperator{\rs}{rs}

\DeclareMathOperator{\codim}{codim}
\DeclareMathOperator{\Av}{Av}
\newcommand{\Ad}{\textup{Ad}}

\DeclareMathOperator{\act}{act}

\DeclareMathOperator{\Frac}{Frac}
\DeclareMathOperator{\coker}{coker}

\def\DD{\mathbb{D}}

\def\GG{\mathbb{G}}

\def\PP{\mathbb{P}}
\def\QQ{\mathbb{Q}}

\def\ZZ{\mathbb{Z}}

\def\frg{\mathfrak{g}}
\def\frt{\mathfrak{t}}

\def\frb{\mathfrak{b}}
\def\frc{\mathfrak{c}}
\newcommand{\frl}{\mathfrak{l}}

\def\calO{\mathcal{O}}
\def\calM{\mathcal{M}}
\def\calA{\mathcal{A}}
\def\calB{\mathcal{B}}
\def\calE{\mathcal{E}}
\def\calF{\mathcal{F}}
\def\calG{\mathcal{G}}
\def\calH{\mathcal{H}}
\def\calL{\mathcal{L}}

\def\calP{\mathcal{P}}

\def\calX{\mathcal{X}}
\def\calY{\mathcal{Y}}

\def\tilp{\widetilde{p}}
\def\tilw{\widetilde{w}}
\def\tilW{\widetilde{W}}
\def\tila{\widetilde{a}}

\newcommand{\till}{\widetilde{\mathfrak{l}}}
\def\tilX{\widetilde{X}}
\def\tilO{\widetilde{\calO}}

\def\tilsigma{\widetilde{\sigma}}

\def\bP{\mathbf{P}}
\def\bI{\mathbf{I}}
\newcommand{\bQ}{\mathbf{Q}}
\newcommand{\bR}{\mathbf{R}}
\newcommand{\bG}{\mathbf{G}}
\newcommand{\bL}{\mathbf{L}}

\newcommand{\unl}{\underline{\mathfrak{l}}}
\newcommand{\unL}{\underline{L}}

\def\Ft{\textup{Ft}}
\def\limD{\underleftarrow{D}}
\def\ind{\textup{ind}}
\def\pro{\textup{pro}}
\def\adj{\textup{ad.}}
\def\xch{\mathbb{X}^*}
\def\xcoch{\mathbb{X}_*}
\def\Sing{\textup{Sing}}
\def\Corr{\textup{Corr}}
\def\Fil{\textup{Fil}}
\def\loc{\textup{loc}}
\def\glob{\textup{glob}}
\def\Ql{\overline{\QQ}_\ell}
\def\one{\mathbf{1}}
\def\aff{\textup{aff}}
\def\Wa{W_{\textup{aff}}}
\def\Tot{\textup{Tot}}
\def\cent{\Ql[\xcoch(T)]^{W}}
\def\Gm{\GG_{m}}
\def\forg{\textup{For}}

\def\Grass{\mathcal{G}r}

\def\Fl{\textup{Fl}}

\def\Hit{\textup{Hit}}
\def\parab{\textup{par}}
\def\ani{\textup{ani}}
\def\Mpar{\mathcal{M}^{\parab}}

\def\MHit{\mathcal{M}^{\Hit}}
\def\Hecke{\calH\textup{ecke}}
\def\Heckep{\Hecke^{\parab}}

\def\tcA{\widetilde{\calA}}
\def\tcArs{\tcA^{\rs}}
\def\AHit{\calA^{\Hit}}
\def\Aa{\calA^{\ani}}
\def\Ah{\calA^{\heartsuit}}
\def\tcB{\widetilde{\calB}}
\def\tcBrs{\widetilde{\calB}^{\rs}}

\def\Spr{\textup{Spr}}
\def\SSpr{\mathcal{S}\textup{pr}}
\def\hSSpr{\widehat{\SSpr}}
\def\Xz{X^{z}}

\def\hM{\widehat{\mathcal{M}}}
\def\hMpar{\hM^{\parab}}
\def\hMparrs{\hM^{\parab,\rs}}
\def\hMHit{\hM^{\Hit}}

\def\hRes{\widehat{\Res}}
\def\hata{\widehat{a}}
\def\hP{\widehat{\mathcal{P}}}
\def\hep{\widehat{\epsilon}}
\def\hH{\widehat{\mathcal{H}}}
\def\hHk{\hH\textup{ecke}}
\def\hHpar{\hH\textup{ecke}^{\parab}}

\def\fpar{f^{\parab}}
\def\fHit{f^{\Hit}}
\def\fQl{f^{\parab}_!\Ql}
\def\fHQl{f^{\Hit}_!\Ql}

\def\hf{\widehat{f}}
\def\hfpar{\widehat{f}^{\parab}}
\def\hfH{\widehat{f}^{\Hit}}

\swapnumbers \theoremstyle{plain}
\newtheorem{theorem}[subsection]{Theorem}
\newtheorem{lemma}[subsection]{Lemma}
\newtheorem{cor}[subsection]{Corollary}
\newtheorem{prop}[subsection]{Proposition}
\newtheorem{conj}[subsection]{Conjecture}

\theoremstyle{definition}

\newtheorem{remark}[subsection]{Remark}

\numberwithin{equation}{section}

\title{The spherical part of the local and global Springer actions}
\author{Zhiwei Yun}
\thanks{Partially supported by the NSF grant DMS-0969470.}
\address{Department of Mathematics, MIT, Cambridge, MA 02139, USA}
\email{zyun@math.mit.edu}
\date{May 2011}
\subjclass[2010]{14H60, 14M15, 20F55, 20G25}

\begin{document}

\begin{abstract}
The affine Weyl group acts on the cohomology (with compact support) of affine Springer fibers (local Springer theory) and of parabolic Hitchin fibers (global Springer theory). In this paper, we show that in both situations, the action of the center of the group algebra of the affine Weyl group (the spherical part) factors through the action of the component group of the relevant centralizers. In the situation of affine Springer fibers, this partially verifies a conjecture of Goresky-Kottwitz-MacPherson and Bezrukavnikov-Varshavsky.

We first prove this result for the global Springer action, and then deduce the result for the local Springer action from that of the global one. The argument strongly relies on the fact that we can deform points on a curve, hence giving an example of using global Springer theory to solve more classical problems.

\end{abstract}

\maketitle

\tableofcontents

\section{Introduction}
Let $k$ be an algebraically closed field. Let $G$ be a reductive algebraic group over $k$. We assume either char$(k)=0$ or char$(k)$ is large with respect to $G$ (see \S\ref{ss:G}). Let $\frg$ be the Lie algebra of $G$. Let $\calB$ be the flag variety of $G$. For $v\in\frg(k)$, the {\em Springer fiber} of $v$ is the closed subvariety $\calB_{v}\subset\calB$ consisting of all Borel subgroups of $G$ whose Lie algebras contain $v$. Classical Springer theory \cite{Spr} gives an action of the Weyl group $W$ of $G$ on the cohomology of $\calB_{v}$. One the other hand, the centralizer $G_{v}$ of $v$ in $G$ acts on $\calB_{v}$, hence on the cohomology of $\calB_{v}$ via its component group $\pi_{0}(G_{v})$. These two symmetries commute with each other:
\begin{equation*}
W\curvearrowright\cohog{*}{\calB_{v}}\curvearrowleft\pi_{0}(G_{v}).
\end{equation*}
However, there is no obvious way to recover the $\pi_{0}(G_{v})$-action on $\cohog{*}{\calB_{v}}$ solely from the $W$-action.

Now we consider the affine situation. Let $F=k((\varpi))$ be the field of formal Laurent series in one variable and $\calO_{F}=k[[\varpi]]$. Let $\Fl_{G}$ be the affine flag variety of $G$ classifying all Iwahori subgroups of $G(F)$. This is an ind-scheme, an infinite union of projective varieties. We may identify $\Fl_{G}(k)$ with $G(F)/\bI$ for a fixed Iwahori subgroup $\bI\subset G(F)$. For any regular semisimple element $\gamma\in\frg(F)$, Kazhdan and Lusztig \cite{KL} defined a closed sub-ind-scheme $\Spr_{\gamma}\subset\Fl_{G}$, called the {\em affine Springer fiber} of $\gamma$. The set $\Spr_{\gamma}(k)$ consists of those Iwahori subgroups whose Lie algebras contain $\gamma$. The ind-scheme $\Spr_{\gamma}$ is a possibly infinite union of projective varieties of dimension expressible in terms of $\gamma$ (see \cite{Be}).

In \cite{Lu}, Lusztig defined an action of the affine Weyl group $\Wa$ on the homology of affine Springer fibers $\Spr_{\gamma}$. We will review this construction in \S\ref{s:loc}, and extend it to an action of the extended affine Weyl group $\tilW=\xcoch(T)\rtimes W$ on both $\homog{*}{\Spr_{\gamma}}$ and $\cohoc{*}{\Spr_{\gamma}}$. On the other hand, the centralizer group $G_{\gamma}(F)$ acts on $\Spr_{\gamma}$, hence induces an action of its component group $\pi_{0}(G_{\gamma}(F))$ on the homology of $\Spr_{\gamma}$. Here we view $G_{\gamma}(F)$ as a group ind-scheme over the base field $k$. These two symmetries on $\homog{*}{\Spr_{\gamma}}$ again commute with each other:
\begin{equation*}
\tilW\curvearrowright\homog{*}{\Spr_{\gamma}}\curvearrowleft\pi_{0}(G_{\gamma}(F)).
\end{equation*}
Similar statement holds for $\cohoc{*}{\Spr_{\gamma}}$. 

A priori, the definition of the $\tilW$-action and the $\pi_{0}(G_{\gamma}(F))$-action has nothing to do with each other. However, as opposed to the situation in classical Springer theory, we expect that the $\pi_{0}(G_{\gamma}(F))$-action be completely determined by the ``central character'' of the $\tilW$-action. The center of the group algebra $\Ql[\tilW]$ is $\cent$ (superscript $W$ means taking $W$-invariants). If one views $\Ql[\tilW]$ as the specialization of the affine Hecke algebra at $q=1$, then $\cent$ is the specialization of the spherical Hecke algebra at $q=1$. For this reason, we shall call $\cent$ the {\em spherical part} of the group algebra $\Ql[\tilW]$. We shall see that there is a canonical algebra homomorphism (see \S\ref{ss:definesigma})
\begin{equation}\label{localcentpi}
\sigma_{\gamma}:\cent\to\Ql[\pi_{0}(G_{\gamma}(F))].
\end{equation}

\begin{conj}[Goresky, Kottwitz and MacPherson \cite{Ko}; independently Bezrukavnikov and Varshavsky \cite{BV}]\label{conj}
For any regular semisimple element $\gamma\in\frg(F)$ and any $i\in\ZZ_{\geq0}$, the spherical part of the $\tilW$-action on $\homog{i}{\Spr_{\gamma}}$ and $\cohoc{i}{\Spr_{\gamma}}$
\begin{equation*}
\cent\to\End(\homog{i}{\Spr_{\gamma}}), \hspace{.5cm}\cent\to\End(\cohoc{i}{\Spr_{\gamma}})
\end{equation*} 
factors through the action of $\pi_{0}(G_{\gamma}(F))$ via the homomorphism \eqref{localcentpi}.
\end{conj}
In fact, the $\pi_{0}(G_{\gamma}(F))$-action factors through a further quotient $\pi_{0}(P_{a(\gamma)})$, where $P_{a(\gamma)}$ is a certain quotient of $G_{\gamma}(F)$ (see \S\ref{ss:localPic}).

The difficulty in proving this conjecture lies in the fact that we do not know an effective way of computing the action of $\cent$: Lusztig's construction of the $\Wa$-action only tells us how each simple reflection acts, but elements in $\cent$ are in general sums of complicated words in simple reflections.

The main purpose of this paper is to prove

\begin{theorem}[Local Main Theorem]\label{th:mainloc} Conjecture \ref{conj} holds for $\cohoc{i}{\Spr_{\gamma}}$.
\end{theorem}

For the homology part of the conjecture, we prove a weaker statement.

\begin{theorem}\label{th:mainhomo}
Under the conditions of Conjecture \ref{conj}, there exists a filtration $\Fil^{p}$ on $\homog{i}{\Spr_{\gamma}}$, stable under both $\tilW$ and $\pi_{0}(G_{\gamma}(F))$, such that the action of $\cent$ on $\Gr^{p}_{\Fil}\homog{i}{\Spr_{\gamma}}$ factors through the action of $\pi_{0}(G_{\gamma}(F))$ via the homomorphism \eqref{localcentpi}.

Moreover, one may choose $\Fil^{p}$ such that $\Gr_{\Fil}^{p}\homog{i}{\Spr_{\gamma}}=0$ unless $0\leq p\leq r$, where $r$ is the split rank of the $F$-torus $G_{\gamma}$.
\end{theorem}

The above conjecture and results have a parahoric version. For each parahoric subgroup $\bP\subset G(F)$ we have the affine partial flag variety $\Fl_{\bP}=G(F)/\bP$ and affine partial Springer fibers $\Spr_{\bP,\gamma}$. The subalgebra $\one_{\bP}\Ql[\tilW]\one_{\bP}\cong\Ql[\xcoch(T)]^{W_{\bP}}$ acts on $\homog{*}{\Spr_{\bP,\gamma}}$ and $\cohoc{*}{\Spr_{\bP,\gamma}}$ (see \S\ref{ss:parahoric}).

\begin{prop} Let $\bP\subset G(F)$ be any parahoric subgroup. Let $\gamma\in\frg(F)$ be a regular semisimple element. If Conjecture \ref{conj} holds for $\homog{i}{\Spr_{\gamma}}$ (resp. $\cohoc{i}{\Spr_{\gamma}}$), then the action of $\cent\subset\one_{\bP}\Ql[\tilW]\one_{\bP}$ on $\homog{i}{\Spr_{\bP,\gamma}}$ (resp. $\cohoc{i}{\Spr_{\bP,\gamma}}$) factors through the action of $\pi_{0}(G_{\gamma}(F))$ via the homomorphism \eqref{localcentpi}.
\end{prop}

In fact, the natural projection $\Spr_{\gamma}\to\Spr_{\bP,\gamma}$ induces a {\em surjection} $\homog{*}{\Spr_{\gamma}}\twoheadrightarrow\homog{*}{\Spr_{\bP,\gamma}}$ and an {\em injection} $\cohoc{*}{\Spr_{\bP,\gamma}}\hookrightarrow\cohoc{*}{\Spr_{\gamma}}$ , which are easily seen to be equivariant under both $\cent$ and $\pi_{0}(G_{\gamma}(F))$ by construction in \S\ref{ss:parahoric}.

Surprisingly, Theorem \ref{th:mainloc} is deduced from its global counterpart, which we state next. Fix a connected smooth projective algebraic curve $X$ over $k$. In \cite[Definition 2.1.1]{GS} we have defined the parabolic Hitchin moduli stack $\Mpar$ classifying quadruples $(x,\calE,\varphi,\calE^{B}_{x})$ where $x\in X$, $\calE$ is a $G$-bundle over $X$, $\varphi$ is a section of the twisted adjoint bundle $\Ad(\calE)\otimes\calO_{X}(D)$ ($D$ is a fixed divisor on $X$ with large degree) and $\calE^{B}_{x}$ is a Borel reduction of $\calE$ at $x$ preserved by the Higgs field $\varphi$. We have the parabolic Hitchin fibration
\begin{equation*}
\fpar:\Mpar\to\AHit\times X
\end{equation*}
recording the characteristic polynomial of $\varphi$ and the point $x$. One of the main results of \cite{GS} is that there exists a natural action of the extended affine Weyl group $\tilW$ on the derived direct image complex $\bR\fQl|_{\Aa\times X}$, where $\Aa\subset\AHit$ is the {\em anisotropic} locus (see \cite[\S 4.10.5]{NgoFL}). In this paper, we will extend this construction to a larger locus $\Ah\times X$, where $\Ah\subset\AHit$ is the {\em hyperbolic} locus (see \cite[\S4.5]{NgoFL}), containing $\Aa$ as an open subset. 

On the other hand, Ng\^o defined a Picard stack $\calP$ over $\Ah$ which acts on $\Mpar$ fiber-wise over $\Ah\times X$. This action induces an action of the sheaf of groups $\pi_{0}(\calP/\Ah)$ (which interpolates the component groups $\pi_{0}(\calP_{a})$ for $a\in\Ah$) on $\bR\fQl$. The study of this action in the case of the usual Hitchin moduli space $\MHit$ leads to the geometric theory of endoscopy, which plays a crucial role in Ng\^o's proof of the Fundamental Lemma \cite{NgoFL}. The idea of relating the $\pi_0(\calP/\Ah)$-action and the $\cent$-action on $\bR\fQl$ was also suggested to the author by Ng\^{o}.

\begin{theorem}[Global Main Theorem]\label{th:mainpar}
For any $i\in\ZZ_{\geq0}$, the spherical part of the $\tilW$-action on the cohomology sheaves of the parabolic Hitchin complex $\bR\fQl$:
\begin{equation*}
\cent\to\End(\bR^{i}\fQl|_{(\Ah\times X)'})
\end{equation*} 
factors through the action of the sheaf $\pi_{0}(\calP/\Ah)$ via a natural homomorphism of sheaves of algebras on $\Ah$:
\begin{equation}\label{globalcentpi}
\sigma:\cent\otimes\const{\Ah}\to\Ql[\pi_{0}(\calP/\Ah)].
\end{equation}
\end{theorem}

Here, $(\Ah\times X)'\subset\Ah\times X$ is any open subset on which a certain codimension estimate holds (see \cite[Proposition 2.6.3, Remark 2.6.4]{GS}). If char$(k)$=0, we may take $(\Ah\times X)'=\Ah\times X$ (see \cite[p.4]{NgoDe}).

A consequence of Theorem \ref{th:mainpar} is
\begin{cor}
For any geometric point $(a,x)\in(\Ah\times X)'$, the action of $\cent$ on $\cohoc{i}{\Mpar_{a,x}}$ factors through the action of $\pi_{0}(\calP_{a})$ via the stalk of the homomorphism \eqref{globalcentpi} at $a$.
\end{cor}

We also have a version of Theorem \ref{th:mainpar} for parahoric Hitchin fibrations. Let us spell out the case of the usual Hitchin fibration $f^{\Hit}:\MHit\to\AHit$. Extending \cite[Theorem 6.6.1]{GS} to the locus $\Ah$, we have an action of $\cent$ on the complex $\bR\fHQl\boxtimes\const{X}$ on $\Ah\times X$.

\begin{theorem}\label{th:Hit}
The action of $\cent$ on the sheaf $\bR^{i}\fHQl\boxtimes\const{X}$ factors through the action of the sheaf of algebras $\Ql[\pi_{0}(\calP/\Ah)]$ via the homomorphism \eqref{globalcentpi}.
\end{theorem}
Note that this theorem is valid on the whole of $\Ah\times X$, rather than just $(\Ah\times X)'$ as in Theorem \ref{th:mainpar}. In the preprint \cite{GSII}, Theorem \ref{th:mainpar} and \ref{th:Hit} were proved over $(\Aa\times X)'$.

\subsection{Application}
In recent work of Bezrukavnikov and Varshavsky, they construct new examples of stable distributions on $p$-adic groups using an affine analog of the relation between character sheaves and the center of the Hecke algebra. When the local field is a function field, a key step in checking the stability of their distributions is Theorem \ref{th:mainloc} and Theorem \ref{th:mainhomo}.

\subsection{Idea of the proof}
The main idea of proving Theorem \ref{th:mainloc} is to view $\calO_{F}$ as the completed local ring of an algebraic curve at one point $x$, and try to deform the point $x$ along the curve $X$. For this, we also need to extend $\gamma\in\frg(F)$ to a $\frg$-valued meromorphic function on $X$. This naturally leads to the consideration of the (parabolic) Hitchin moduli stack, hence leading to Theorem \ref{th:mainpar}.

To prove Theorem \ref{th:mainpar}, we first prove Theorem \ref{th:Hit}. The $\cent$-action on $\bR\fHQl\boxtimes\const{X}$ can be thought of as a family of $\cent$-actions on $\bR\fHQl$ indexed by $x\in X$. Homotopy invariance guarantees that the effect of this action on $\bR^{i}\fHQl$ is independent of $x$. Then we only need to check that the $\cent$-action at a general point $x\in X$ does factor through $\pi_{0}(\calP/\Ah)$, which is clear from the construction in \cite{GS}.

To deduce Theorem \ref{th:mainpar} from Theorem \ref{th:Hit}, we simultaneously deform the point of Borel reduction (which is contained in the moduli problem of $\Mpar$) and the point of Hecke modification (which gives the $\cent$-action). This way we get a result about the $\cent$-action on $\bR\fQl\boxtimes\const{X}$ analogous to Theorem \ref{th:Hit}, but this time our complex lives on $\Ah\times X^{2}$. Restricting to the diagonal $\Ah\times\Delta(X)$, we get the desired factorization in Theorem \ref{th:mainpar}. The idea behind this argument is reminiscent of Gaitsgory's construction of the center of the affine Hecke algebra via nearby cycles (see \cite{Ga}).

Finally, we need to deduce Theorem \ref{th:mainloc} from Theorem \ref{th:mainpar}. We argue that for every affine Springer fiber $\Spr_{\gamma}$, its compactly supported cohomology appears inside the compactly supported cohomology of a certain rigidified parabolic Hitchin fiber, and this inclusion respects the various symmetries. This fact follows from a detailed analysis of Ng\^o's product formula.

We remark that using a parabolic version of Ng\^o's Support Theorem (see \cite{GSLD}), one can deduce a weaker version of Theorem \ref{th:mainpar}, namely the {\em semisimplification} of the $\cent$-action factors through the $\pi_{0}(\calP/\Aa)$-action on $\bR\fQl|_{\Aa\times X}$. 

The proof described above shows that the global Springer theory developed in \cite{GS} can be useful in solving more classical problems about affine Springer fibers.

\subsection{Convention} 
Throughout the paper, $k$ will be an algebraically closed field. The stacks on which we talk about sheaves are of the form $[X/A]$, where $X$ is an algebraic space, locally of finite type over $k$, and $A$ is a linear algebraic group over $k$ which acts on $X$. All complexes of sheaves will be objects in the derived category of $\Ql$-complexes in the \'etale topology. See \cite[\S1.1]{WeilII} for the case of schemes of finite type and \cite{O} for the case of stacks. See Appendix \ref{a:lft} for the convention for sheaves on algebraic spaces which are locally of finite type over $k$. {\bf All sheaf-theoretic functors are understood to be derived without putting $\bR$ or $\bL$ in the front.} For a morphism $f:\calX\to\calY$ between stacks, we use $\DD_{\calX/\calY}$ or $\DD_{f}$ to denote the relative dualizing complex $f^{!}\const{\calY}$. The homology complex of $f$ is defined as
\begin{equation*}
\homo{*}{\calX/\calY}:=f_{!}\DD_{\calX/\calY}.
\end{equation*}
In particular, if $\calY=\Spec k$, we write $\homog{*}{\calX}$ for $\homo{*}{\calX/\Spec k}$.

\subsection{Notations for $G$}\label{ss:G} Let $G$ be a reductive algebraic group over $k$. Fix a maximal torus $T$ of $G$ and a Borel $B$ containing $T$. Let $\frg,\frb,\frt$ be the Lie algebras of $G,B,T$ respectively. Let $(\xch(T),\Phi,\xcoch(T),\Phi^{\vee})$ be the based root and coroot systems determined by $(G,B,T)$. Let $W$ be the Weyl group. Let $\frc=\frg\sslash G=\frt\sslash W$ be the adjoint quotient of $\frg$ in the GIT sense. 

We now make precise the assumption on char $(k)$. Let $h$ be one plus the sum of coefficients of the highest root of $G$ written in terms of simple roots. We assume either char$(k)=0$ or char$(k)>2h$. We impose this condition because we would like to make sure that the Kostant section $\epsilon:\frc\to\frg$ exists, see \cite[\S1.2]{NgoFL}. 

The extended affine Weyl group and the affine Weyl group are defined as
\begin{equation*}
\tilW:=\xcoch(T)\rtimes W; \hspace{1cm} \Wa:=\ZZ\Phi^{\vee}\rtimes W
\end{equation*}
where $\ZZ\Phi^{\vee}\subset\xcoch(T)$ is the coroot lattice. The affine Weyl group $\Wa$ is a Coxeter group with the set of simple reflections $\Delta_{\aff}$. There is an exact sequence
\begin{equation}\label{Wa}
1\to\Wa\to\tilW\to\Omega\to1
\end{equation}
with $\Omega=\xcoch(T)/\ZZ\Phi^{\vee}$ a finite abelian group.

\section{Local Springer action}\label{s:loc}
In this section, we will explain all the ingredients that go into the statement of Conjecture \ref{conj}. 

\subsection{Loop group}\label{ss:aff} Let $F=k((\varpi))$ and $\calO_{F}=k[[\varpi]]\subset F$. For a scheme $X$ over $F$, we use $\hRes^{F}_{k}X$ to denote the functor \{$k$-algebras\}$\to$\{Sets\}:
\begin{equation*}
(\hRes^{F}_{k}X)(R)=X(R((\varpi))).
\end{equation*}
Similarly, if $\calX$ is defined over $\calO_{F}$,we define $\hRes^{\calO_{F}}_{k}\calX$ to be the functor
\begin{equation*}
(\hRes^{\calO_{F}}_{k}\calX)(R)=\calX(R[[\varpi]]).
\end{equation*}
The functor $\hRes^{F}_{k}(G\otimes_{k}F)$ is represented by a group ind-scheme $G((\varpi))$ over $k$, called the {\em loop group} of $G$. The standard Iwahori subgroup $\bI\subset G(\calO_{F})$ is the preimage of $B(k)$ under the evaluation map $G(\calO_{F})\to G(k)$. Standard parahoric subgroups $\bP\supset\bI$ are in bijection with proper subsets of $\Delta_{\aff}$. The parahoric $G(\calO_{F})$ is denoted by $\bG$.

For each parahoric $\bP$ there is a smooth $\calO_{F}$-model $\calG_{\bP}$ of $G\otimes_{k}F$, the Bruhat-Tits group scheme, such that $\calG_{\bP}(\calO_{F})=\bP$. The {\em Lie algebra} $\Lie\bP$ of $\bP$ is, by definition, the Lie algebra of $\calG_{\bP}$, hence a finite free $\calO_{F}$-module. Let $L_{\bP}$ be maximal reductive quotient of the special fiber of $\calG_{\bP}$. This is a connected reductive group over $k$. The functor $\hRes^{\calO_{F}}_{k}\calG_{\bP}$ is represented by a group scheme over $k$. The affine partial flag variety of type $\bP$ is the ind-scheme
\begin{equation*}
\Fl_{\bP}=G((\varpi))/\hRes^{\calO_{F}}_{k}\calG_{\bP}.
\end{equation*}
The set $\Fl_{\bP}(k)$ classifies all parahoric subgroups of $G(F)$ which are conjugate to $\bP$.

Having fixed an Iwahori $\bI$, the exact sequence \eqref{Wa} admits a section $\Omega\to\tilW$ whose image $\Omega_{\bI}$ is the stabilizer of the fundamental alcove corresponding to $\bI$ in the reduced building of $G(F)$. The group $\Omega_{\bI}\subset\tilW$ can also be identified with $N(\bI)/\bI$ where $N(\bI)$ is the normalizer of $\bI$ in $G(F)$.

\subsection{Affine Springer fibers} Let $\gamma\in\frg(F)$ be a regular semisimple element. The sub-ind-scheme $\hSSpr_{\bP,\gamma}\subset G((\varpi))$ is defined to have $R$-points
\begin{equation*}
\hSSpr_{\bP,\gamma}(R)=\{g\in G(R((\varpi)))|\Ad(g^{-1})\gamma\in R\otimes_{k}\Lie\bP\}.
\end{equation*}
Clearly the right $\hRes^{\calO_{F}}_{k}\calG_{\bP}$-action on $G((\varpi))$ preserves $\hSSpr_{\bP,\gamma}$, hence we can define the quotient
\begin{equation*}
\SSpr_{\bP,\gamma}=\hSSpr_{\bP,\gamma}/\hRes^{\calO_{F}}_{k}\calG_{\bP}\subset\Fl_{\bP}.
\end{equation*}
As defined above $\SSpr_{\bP,\gamma}$ is highly non-reduced. Its reduced structure $\Spr_{\bP,\gamma}$ is a scheme locally of finite type, see \cite{KL}. We call $\Spr_{\bP,\gamma}$ the {\em affine Springer fiber of $\gamma$ with type $\bP$}. When $\bP=\bI$, we often omit $\bI$ from subscripts.

\subsection{Lusztig's construction of the $\Wa$-action in \cite{Lu}}\label{ss:L}
Let $\frl_{\bP}$ be the Lie algebra of $L_{\bP}$ and $\till_{\bP}$ be the Grothendieck simultaneous resolution of $\frl_{\bP}$. We have a Cartesian diagram
\begin{equation}\label{SprCart}
\xymatrix{\SSpr_{\gamma}\ar[r]^{\ev_{\gamma}}\ar[d]^{\nu_{\bP}} & [\till_{\bP}/L_{\bP}]\ar[d]^{\pi_{\bP}}\\
\SSpr_{\bP,\gamma}\ar[r]^{\ev_{\bP,\gamma}} & [\frl_{\bP}/L_{\bP}]}
\end{equation}

Apply proper base change to the diagram \eqref{SprCart}, we get
\begin{equation*}
\nu_{\bP*}\DD_{\Spr_{\gamma}}=\ev_{\bP,\gamma}^{!}\pi_{\bP*}\DD_{[\till_{\bP}/L_{\bP}]}.
\end{equation*}
By classical Springer theory for the Lie algebra $\frl_{\bP}$, there is a $W_{\bP}$-action on $\pi_{\bP,*}\DD_{[\till_{\bP}/L_{\bP}]}$ (to see this, we may identify $\DD_{[\till_{\bP}/L_{\bP}]}$ with the constant sheaf on $[\till_{\bP}/L_{\bP}]$). Therefore $\nu_{\bP*}\DD_{\Spr_{\gamma}}$ also carries a $W_{\bP}$-action. Since $\nu_{\bP,\gamma}$ is proper, $\homog{*}{\Spr_{\gamma}}=\bR\Gamma_{c}(\Spr_{\bP,\gamma},\nu_{\bP*}\DD_{\Spr_{\gamma}})$ also carries a $W_{\bP}$-action.

For $\bP\subset\bQ$, Lusztig then argues that the $W_{\bQ}$-action on $\homog{*}{\Spr_{\gamma}}$ restricts to the $W_{\bP}$-action defined both as above. Therefore, these $W_{\bP}$-actions generate a $\Wa$-action on $\homog{*}{\Spr_{\gamma}}$.

Replacing the dualizing complexes by the constant sheaves in the above discussion, we obtain an action of $\Wa$ on $\cohoc{*}{\Spr_{\gamma}}$.

\subsection{The $\Omega_{\bI}$-action}\label{ss:O} Viewing $\Omega_{\bI}$ as the quotient $N(\bI)/\bI$, we get a natural action of $\Omega_{\bI}$ on $\Fl=\Fl_{\bI}$ by right multiplication. We denote this action by $\omega\mapsto R_{\omega}$ for $\omega\in\Omega_{\bI}$. This action preserves $\Spr_{\gamma}$ (because $\Ad(g)\bI=\Ad(g\omega)\bI$ for any $\omega\in N(\bI)$), hence $\omega\in\Omega_{\bI}$ acts on $\homog{*}{\Spr_{\gamma}}$ from the left:
\begin{equation*}
R^{-1}_{\omega,*}:\homog{*}{\Spr_{\gamma}}\to\homog{*}{\Spr_{\gamma}}.
\end{equation*}
Similarly, right multiplication by $\omega\in N(\bI)$ sends one standard parahoric $\bP$ to another standard parahoric $\omega^{-1}\bP\omega$, and gives an isomorphism $\Fl_{\bP,\gamma}\isom\Fl_{\omega^{-1}\bP\omega,\gamma}$. We have a commutative diagram
\begin{equation}\label{Ro}
\xymatrix{\Spr_{\gamma}\ar[r]^{R_\omega}\ar[d]^{\nu_{\bP}} & \Spr_{\gamma}\ar[d]^{\nu_{\omega^{-1}\bP\omega}}\\
\Spr_{\bP,\gamma}\ar[r]^{R_\omega} & \Spr_{\omega^{-1}\bP\omega,\gamma}}
\end{equation}
for any $\omega\in\Omega_{\bI}$. This implies that $R^{-1}_{\omega,*}$ intertwines the action of $W_{\omega^{-1}\bP\omega}$ and of $W_{\bP}$ on $\homog{*}{\Spr_{\gamma}}$, via the isomorphism $\Ad(\omega):W_{\omega^{-1}\bP\omega}\to W_{\bP}$. Similar remarks apply to $\cohoc{*}{\Spr_{\gamma}}$. Summarizing, we get

\begin{theorem} Lusztig's construction in \S\ref{ss:L} and the $\Omega_{\bI}$-action in \S\ref{ss:O} together generate a $\tilW$ action on both $\homog{*}{\Spr_{\gamma}}$ and $\cohoc{*}{\Spr_{\gamma}}$.
\end{theorem}

\subsection{The parahoric version}\label{ss:parahoric} For each standard parahoric $\bP$, let $W_{\bP}\subset\Wa$ be the finite Weyl group of the Levi quotient $L_{\bP}$. The Cartesian diagram \eqref{SprCart} and proper base change implies that
\begin{equation*}
\homog{*}{\Spr_{\bP,\gamma}}=\homog{*}{\Spr_{\gamma}}^{W_{\bP}},\hspace{1cm} \cohoc{*}{\Spr_{\bP,\gamma}}=\cohoc{*}{\Spr_{\gamma}}^{W_{\bP}}.
\end{equation*}
In fact, the constant sheaf on $\frt_{\bP}$ is the $W_{\bP}$-invariants of the Springer sheaf $\pi_{\bP,*}\const{\till_{\bP}}$. Let
\begin{equation*}
\one_{\bP}=\frac{1}{\# W_{\bP}}\sum_{w\in W_{\bP}}w\in\Ql[\tilW]
\end{equation*}
be an idempotent. Then the subalgebra $\one_{\bP}\Ql[\tilW]\one_{\bP}$ acts on $\homog{*}{\Spr_{\bP,\gamma}}$ and $\cohoc{*}{\Spr_{\bP,\gamma}}$.

\subsection{Local Picard}\label{ss:localPic} We follow \cite[\S3.3]{NgoFL} in this subsection.
For $\gamma\in\frg(\calO_{F})$, let $a(\gamma)$ be its image in $\frc(\calO_{F})$.   

Recall the $F$-torus $G_{\gamma}$ admits a smooth model $J_{a(\gamma)}$ over $\Spec\calO_{F}$ which canonically only depends on $a(\gamma)$. This is the regular centralizer group scheme defined in \cite[\S3]{NgoHit}. Let $\calP_{a(\gamma)}$ be the affine Grassmannian of the $\Spec\calO_{F}$-group scheme $J_{a(\gamma)}$ (see \cite[\S 3.8]{NgoFL}): 
\begin{equation*}
\calP_{a(\gamma)}=\Grass_{J_{a(\gamma)}}:=\hRes^{F}_{k}J_{a(\gamma)}/\hRes^{\calO_{F}}_{k}J_{a(\gamma)}=\hRes^{F}_{k}G_{\gamma}/\hRes^{\calO_{F}}_{k}J_{a(\gamma)}.
\end{equation*}
For a $k$-algebra $R$, $\calP_{a(\gamma)}(R)$ is the set of isomorphism classes of $J_{a(\gamma)}$-torsors over $\Spec R[[\varpi]]$ together with a trivialization over $\Spec R((\varpi))$. Since $J_{a(\gamma)}$ is commutative, $\calP_{a(\gamma)}$ has a group ind-scheme structure. The action of $\hRes^{F}_{k}G_{\gamma}$ on $\Spr_{\bP,\gamma}$ factors through $\calP_{a(\gamma)}$.

We have the finite type N\'eron model $J^{\flat}_{a(\gamma)}$ of $J_{a(\gamma)}$ (see \cite[\S3.8]{NgoFL}). We define $\calP^{\flat}_{a(\gamma)}$ similarly using $J^{\flat}_{a(\gamma)}$ instead of $J_{a(\gamma)}$. By \cite[Lemme 3.8.1]{NgoFL}, the reduced structure of $\calP^\flat_{a(\gamma)}$ is a free abelian group $\Lambda_{a(\gamma)}$. Let $P_{a(\gamma)}\hookrightarrow\calP_{a(\gamma)}$ be the preimage of $\Lambda_{a(\gamma)}\hookrightarrow\calP^{\flat}_{a(\gamma)}$, then we have an exact sequence
\begin{equation*}
1\to R_{a(\gamma)}\to P_{a(\gamma)}\to\Lambda_{a(\gamma)}\to1.
\end{equation*}
where the kernel $R_{a(\gamma)}$ is an affine commutative group scheme of finite type.

\subsection{Definition of the map $\sigma_{\gamma}$ in \eqref{localcentpi}}\label{ss:definesigma}
Consider the following diagram
\begin{equation}\label{localcameral}
\xymatrix{\Spec\tilO_{F}\ar[dr]\ar[r] & \Spec\calO_{F,a}\ar[r]\ar[d] & \frt\ar[d]\\
& \Spec\calO_{F}\ar[r]^(.6){a(\gamma)} & \frc}.
\end{equation}
where the square on the right is Cartesian by definition. The morphism $\Spec\calO_{F,a}\to\Spec\calO_{F}$ is called the {\em local cameral cover}. The ring $\tilO_{F}$ is the normalization of $\calO_{F,a}$. Choose a component $\Spec\tilO^{!}_{F}\subset\Spec\tilO_{F}$. Let $W^{!}\subset W$ be the stabilizer of $\Spec\tilO^{!}_{F}$. According to \cite[Prop. 3.9.2]{NgoFL} (or rather its dual statement), the choice of $\Spec\tilO^{!}_{F}$ allows us to define a canonical surjection
\begin{equation}\label{xcochgamma}
\xcoch(T)\twoheadrightarrow\xcoch(T)_{W^{!}}=\pi_{0}(\hRes^{F}_{k}G_{\gamma})\twoheadrightarrow\pi_{0}(P_{a(\gamma)}).
\end{equation}
If we change the choice of $\tilO^{!}_{F}$, the above map will differ by an action of $W$ on $\xcoch(T)$. In particular, taking the group algebras in \eqref{xcochgamma} and restricting to $\cent$, the map
\begin{equation*}
\sigma_{\gamma}:\cent\to\Ql[\hRes^{F}_{k}G_{\gamma}]\to\Ql[\pi_{0}(P_{a(\gamma)})]
\end{equation*}
is independent of any choice. The first of map above is the map in \eqref{localcentpi}.

\begin{prop}[Local constancy]\label{p:constancy}
Fix a regular semisimple $\gamma\in\frg(F)$ with $a(\gamma)\in\frc(\calO_{F})$. There is an integer $N>0$ such that for any $\gamma'\in\frg(F)$ with $a(\gamma')\in\frc(\calO_{F})$ and $a(\gamma')\equiv a(\gamma)\mod\varpi^{N}$, there are isomorphisms
\begin{eqnarray*}
\iota_{\calP}&:&\calP_{\gamma}\isom \calP_{\gamma'};\\
\iota&:&\SSpr_{\gamma}\isom\SSpr_{\gamma'}
\end{eqnarray*}
such that $\iota$ is equivariant under the $\calP_{\gamma}$ and $\calP_{\gamma'}$ actions via $\iota_{\calP}$. Moreover, the isomorphism $\iota$ can be chosen so that both $\iota^{*}:\cohoc{*}{\Spr_{\gamma}}\to\cohoc{*}{\Spr_{\gamma'}}$ and $\iota_{*}:\homog{*}{\Spr_{\gamma}}\isom\homog{*}{\Spr_{\gamma'}}$ are $\tilW$-equivariant.
\end{prop}
\begin{proof}
We first deal with the case $a(\gamma)=a(\gamma')$. Since the field $F$ has dimension $\leq1$, $\cohog{1}{F,A}=0$ for any torus $A$ over $F$ (see \cite[Ch. X, end of \S7]{Serre}). In particular, if $\gamma$ and $\gamma'$ have the same image in $\frc^{\rs}(F)$, they are conjugate by an element $g\in G(F)$, and the required isomorphisms are given by $\Ad(g)$.

By the above discussion, we may assume $\gamma$ is the Kostant section of $a(\gamma)$ (see \cite[\S1.2]{NgoFL}); similarly we may assume $\gamma'$ is the constant section of $a(\gamma')$.

We need a variant of \cite[Lemme 3.5.3]{NgoFL} with $\frg(\calO_{F})=\Lie\bG$ replaced by $\Lie\bI$. For this, one only needs to use \cite[Lemme 2.4.3]{NgoFL} in place of \cite[Lemme 2.1.1]{NgoFL} in the argument. This variant of \cite[Lemme 3.5.3]{NgoFL} shows that $\SSpr_{\gamma}$ canonically depends only on the $\calO_{F}$-group scheme $G_{\gamma}=J_{a(\gamma)}$ (recall $\gamma\in\frg(\calO_{F})$ comes from the Kostant section).

By \cite[Lemme 3.5.2]{NgoFL}, there is an integer $N>0$ such that the local cameral covers $\calO_{F,a(\gamma)}$ and $\calO_{F,a(\gamma')}$ are $W$-equivariantly isomorphic as $\calO_{F}$-modules. By \cite[Lemme 3.5.4]{NgoFL}, there exists $g\in G(\calO_{F})$ such that $\Ad(g)G_{\gamma}=G_{\gamma'}$ as subgroups of $G\otimes_{k}\calO_{F}$. The isomorphism $\iota_{\calP}$ is induced from $\Ad(g)$.  The left translation $g:\Fl\to\Fl$ then induces an isomorphism $\iota:\Spr_{\gamma}\isom\Spr_{\gamma'}$ intertwining the actions of $\calP_{a(\gamma)}$ and $\calP_{a(\gamma')}$. This proves the first statement of the Proposition.

To prove that $\iota_{*}$ and $\iota^{*}$ are $\tilW$-equivariant, one only needs to notice that under the left translation by $g\in G(\calO_{F})$, the diagrams \eqref{SprCart} and \eqref{Ro} for $\gamma$ map isomorphically to the corresponding diagrams for $\gamma'$, at least after replacing the $\SSpr$'s by their reduced structures $\Spr$. For the diagram \eqref{SprCart}, we remark that $\Spr_{\bP,\gamma}$ maps isomorphically to $\Spr_{\bP,\gamma'}$ under $g$, since they are images of $\Spr_{\gamma}$ and $\Spr_{\gamma'}$ under the projection $\Fl\to\Fl_{\bP}$. Since these diagrams determine the $\tilW$-action by construction, the $\tilW$-equivariance follows.

\end{proof}

\section{Global Springer action: extension to the hyperbolic locus}
In this section, we extend the $\tilW$-action on $\fQl|_{\Aa\times X}$ constructed in \cite{GS} from the anisotropic locus $\Aa$ to the {\em hyperbolic locus} $\Ah$.

\subsection{The Hitchin moduli stack}

We first recall the definition of the Hitchin moduli stack. Fix a divisor $D=2D'$ on $X$ with $\deg(D)\geq 2g_{X}$. The {\em Hitchin moduli stack} $\MHit=\MHit_{X,G,D}$ assigns to a $k$-scheme $S$ the groupoid of {\em Hitchin pairs} $(\calE,\varphi)$ where
\begin{itemize}
\item $\calE$ is a (right) $G$-torsor over $X\times S$;
\item $\varphi$ is a section of the vector bundle $\Ad(\calE)\otimes\calO_{X}(D)$, where $\Ad(\calE)=\calE\twtimes{G}\frg$ is the adjoint bundle over $X\times S$. 
\end{itemize}
It is well-known that $\MHit$ is an algebraic stack.

Let $\frc_{D}$ be the affine space bundle $\Tot^{\times}(D)\twtimes{\Gm}\frc$ where $\Tot^{\times}(D)$ is the $\Gm$-torsor associated to the line bundle $\calO_{X}(D)$. Let $\AHit=\cohog{0}{X,\frc_{D}}$ be the Hitchin base. We have the {\em Hitchin fibration}
\begin{equation*}
\fHit: \MHit\to \AHit
\end{equation*}
which assigns $(\calE,\varphi)$ the ``invariant polynomials'' of $\varphi$. Recall from \cite[\S4.5]{NgoFL} that there is an open subset $\Ah\subset\AHit$ consisting of those sections $a:X\to\frc_{D}$ which generically lies in the regular semisimple locus $\frc^{\rs}_{D}$. We call $\Ah$ the {\em hyperbolic locus} of $\AHit$.

\subsection{Rigidified Hitchin moduli space}

Fix a point $z\in X(k)\backslash D$. Let $\calA^{z}\subset\AHit$ be the open subset consisting of $a:X\to\frc_{D}$ such that $a(z)\in\frc^{\rs}_{D}$. Clearly $\calA^{z}\subset\Ah$. Define $\calB$ by the Cartesian diagram
\begin{equation*}
\xymatrix{\calB\ar[r]\ar[d] & \frt^{\rs}\ar[d]\\
\calA^{z}\ar[r]^{a\mapsto a(z)} & \frc^{\rs}}
\end{equation*}
So $\calB\to\calA^{z}$ is a $W$-torsor. 

Let $\hMHit$ be the functor which assigns to a $k$-scheme $S$ the groupoid of tuples $(\calE,\varphi,t_{z},\iota_{z})$ where
\begin{itemize}
\item $(\calE,\varphi)\in\MHit(S)$;
\item Let $a:X\times S\to\frc_{D}$ be the image of $(\calE,\varphi)$ under $\fHit$, and $a(z):\{z\}\times S\to\frc$. We require that the image of $a(z)$ lies in $\frc^{\rs}$. Moreover, $t_{z}:S\to\frt^{\rs}$ is a lifting of $a(z)$.
\item $\iota_{z}:(\calE,\varphi)|_{\{z\}\times S}\isom(G\times S,t_{z})$ is an isomorphism of Hitchin pairs over $\{z\}\times S$.
\end{itemize}  

Note that $T$ acts on $\hMHit$ by changing the trivialization $\iota_{z}$, because it is the centralizer of $t_{z}$.

\begin{lemma}\label{l:sch}
The functor $\hMHit$ is represented by an algebraic space which is locally of finite type and smooth over $k$.
\end{lemma}
\begin{proof}
Forgetting $t_{z}$ and $\iota_{z}$ we get a morphism $\hMHit\to\MHit$ which is a $N_{G}(T)$-torsor. In particular, the forgetful morphism is of finite type. Since $\MHit$ is an algebraic stack which is locally of finite type and smooth over $k$, so is $\hMHit$.

It remains to show that the automorphism group of any geometric point $(\calE,\varphi,t_{z},\iota_{z})\in\hMHit(K)$ is trivial ($K\supset k$ being any algebraically closed field). In fact, in \cite[Proposition 4.11.2]{NgoFL} it is shown that
\begin{equation*}
\Aut(\calE,\varphi)\subset\cohog{0}{X_{K},J^{\flat}_{a}}.
\end{equation*}
where $J^{\flat}_{a}$ is the (finite type ) N\'eron model of the regular centralizer group scheme $J_{a}$ over $X_{K}$. Let $q_{a}:X_{a}\to X_{K}$ be the cameral cover of $X_{K}$, then by \cite[Corollaire 4.8.1]{NgoFL}, we have
\begin{equation*}
\cohog{0}{X_{K},J^{\flat}_{a}}=\cohog{0}{X_{a},T}^{W}\subset (q_{a}^{-1}(z)\times T)^{W}=J^{\flat}_{a,z}=\Aut((\calE,\varphi)_{z})
\end{equation*}
the last equality holds because $a(z)\in\frc^{\rs}(K)$. Therefore there is no nontrivial automorphism of $(\calE,\varphi)$ which preserves $\iota_{z}$. This shows that the automorphism group of the triple $(\calE,\varphi,t_{z},\iota_{z})$ is trivial.
\end{proof}

We still have the Hitchin fibration (which is no longer proper even over the anisotropic locus)
\begin{equation*}
\hfH:\hMHit\to\calB.
\end{equation*}

\subsection{Rigidified parahoric Hitchin moduli space} For each standard parahoric subgroup $\bP\subset G(F)$, we have defined in \cite[Definition 4.3.3]{GS} an algebraic stack $\calM_{\bP}$ classifying Hitchin pairs $(\calE,\varphi)$ together with a $\bP$-level structure $\calE^{\bP}_{x}$ of $\calE$ at a varying point $x\in X$ such that $\varphi$ is compatible with $\calE^{\bP}_{x}$. For the precise meaning of ``$\bP$-level structure'' and ``compatible'' we refer the readers to \cite[\S4.3]{GS}. 

A particular case is when $\bP=\bI$, then $\Mpar:=\calM_{\bI}$ is called the {\em parabolic Hitchin moduli stack}, which classifies quadruples $(\calE,\varphi,x,\calE^{B}_{x})$ where $(\calE,\varphi)\in\MHit$, $x\in X$ and $\calE^{B}_{x}$ is a $B$-reduction of $\calE$ at $x$ such that $\varphi(x)\in\Ad(\calE^{B}_{x})\otimes\calO(D)_{x}$. Another special case is when $\bP=\bG$, then $\calM_{\bG}=\MHit\times X$. For each $\bP$, we have the parahoric Hitchin fibration 
\begin{equation*}
f_{\bP}:\calM_{\bP}\to\AHit\times X.
\end{equation*}

Let $\hM_{\bP}$ be the stack classifying data $(\calE,\varphi,x,\calE^{\bP}_{x},t_{z},\iota_{z})$ where 
\begin{itemize}
\item $(\calE,\varphi,x,\calE^{\bP}_{x})\in\calM_{\bP}$;
\item $x$ is disjoint from $z$ and $a(z)\in\frc^{\rs}$ ($a=f^{\Hit}(\calE,\varphi)\in\AHit$), $t_{z}\in\frt^{\rs}$ lifts $a(z)$;
\item $\iota_{z}$ is an isomorphism $(\calE,\varphi)|_{z}\isom (G,t_{z})$ of Hitchin pairs at $\{z\}$ (cf. the definition of $\hMHit$).
\end{itemize}
Parallel to Lemma \ref{l:sch}, $\hM_{\bP}$ is represented by an algebraic space which is locally of finite type and smooth over $k$. We also have morphisms
\begin{equation}\label{hf}
\hf_{\bP}:\hM_{\bP}\to\calB\times\Xz
\end{equation}
where $\Xz=X\backslash\{z\}$. When $\bP=\bI$, we usually write the morphism \eqref{hf} as
\begin{equation*}
\hfpar:\hMpar\to\calB\times\Xz.
\end{equation*}
For two parahorics $\bQ\subset\bP$, we have a forgetful morphism over $\calB\times\Xz$:
\begin{equation*}
\widehat{\forg}^{\bP}_{\bQ}:\hM_{\bQ}\to\hM_{\bP}
\end{equation*}

\subsection{Construction of the $\tilW$-action}\label{ss:action} In this subsection we construct a $\tilW$-action on the direct image complex $\hfpar_{!}\Ql$. Since $T$ acts on $\hMpar$ and $\hfpar$ is $T$-invariant, we can view $\hfpar_{!}\Ql$ as an object in $\ind D^{b}_{T}(\calB\times\Xz)$, where $D^{b}_{T}(\calB\times\Xz)$ is the $T$-equivariant derived category of $\Ql$-complexes on $\calB\times\Xz$ (with trivial $T$-action). The construction of the $\tilW$-action is completely parallel to the case of $\fpar$ treated in \cite[\S5.1]{GS} and the affine Springer fiber case in \S\ref{ss:L}.

For each standard parahoric $\bP$, we have a Cartesian diagram
\begin{equation}\label{MCart}
\xymatrix{\hMpar\ar[r]^{\ev^{\bP}_{\bI}}\ar[d]^{\widehat{\forg}^{\bP}_{\bI}} & [\underline{\till_{\bP}}/\unL_{\bP}]_{D}\ar[d]\\
\hM_{\bP}\ar[r]^{\ev_{\bP}} & [\unl_{\bP}/\unL_{\bP}]_{D}}
\end{equation}
Here $\unL_{\bP}$ is a group scheme over $X$ which is an inner form of $L_{\bP}$. For precise definition, see \cite[Equation (4.1)]{GS}. Similarly, we have the twisted versions $\unl_{\bP}$ and $\underline{\till}_{\bP}$ of the Lie algebra $\frl_{\bP}$ and its Grothendieck resolution $\till_{\bP}$. So $[\underline{\till_{\bP}}/\unL_{\bP}]$ and $[\unl_{\bP}/\unL_{\bP}]$ are stacks over $X$ with natural $\Gm$-actions by dilation. Adding a subscript $D$ means applying the twisted product $(-)\twtimes{\Gm}_{X}\Tot^{\times}(D)$ to these stacks.

With this diagram, we can define a $W_{\bP}$-action on $\widehat{\forg}^{\bP}_{\bI,*}\Ql\in D^{b}(\hM_{\bP})$ similarly as in \S\ref{ss:L} or \cite[Construction 5.1.1]{GS}. Therefore we get a $W_{\bP}$-action on the ind-object $\hfpar_{!}\Ql=\hf_{\bP,!}\widehat{\forg}^{\bP}_{\bI,*}\Ql$. As in the proof of \cite[Theorem 5.1.2]{GS}, these actions for various $\bP$ are compatible, and they together give an action of $\Wa$ on $\hfpar_{!}\Ql$.

On the other hand, $\Omega_{\bI}$ still acts on $\hMpar$ on the right, lifting its action on $\Mpar$ in \cite[Corollary 4.3.4]{GS}. This gives an $\Omega_{\bI}$-action on $\hfpar_{!}\Ql$. Putting together with the $\Wa$-action, we get a $\tilW$-action on $\hfpar_{!}\Ql$.

The diagram \eqref{MCart} implies
\begin{equation*}
\hf_{\bP,!}\Ql=(\hfpar_{!}\Ql)^{W_{\bP}}.
\end{equation*}
Therefore we get an action of $\one_{\bP}\Ql[\tilW]\one_{\bP}$ on $\hf_{\bP,!}\Ql$.

\subsection{Hecke correspondences} In \cite[\S3]{GS}, we also have a construction of the $\tilW$-action on $\fQl$ via Hecke correspondences. Here we extend the construction to the case of $\hfpar_{!}\Ql$.

Recall we have a Hecke correspondence $\Heckep$, which is a self-correspondence of $\Mpar$ over $\AHit\times X$. Over the locus $(\Ah\times X)^{\rs}$, there is a $\tilW$-action on $\Mpar|_{(\Ah\times X)^{\rs}}$. For each $\tilw\in\tilW$, the closure $\calH_{\tilw}$ of the graph of the $\tilw$-action is a closed subspace of $\Heckep$.

Let $\oright{h}$ be the second projection from $\Heckep$ or $\calH_{\tilw}$ to $\Mpar$. Let
\begin{eqnarray*}
\hHpar &=& \Heckep\times_{\oright{h},\Mpar}\hMpar;\\
\hH_{\tilw} &=& \calH_{\tilw}\times_{\oright{h},\Mpar}\hMpar.
\end{eqnarray*}
Then $\hHpar$ and $\hH_{\tilw}$ can be viewed as self-correspondences of $\hMpar$ over $\calB\times\Xz$. In fact, $\hHpar$ parametrizes two Hitchin pairs with Borel reductions at a point $x\neq z$, an isomorphism of these Hitchin pairs on $X\backslash\{x\}$ and a rigidification $\iota_{z}$ of the second Hitchin pair at $z$ (which then automatically gives a rigidification of the first Hitchin pair at $z$).

\subsection{The subset $(\calB\times\Xz)'$}\label{ss:BXprime} 
On the scheme $\Ah\times X$, we have an upper semi-continuous function $\delta$ given by the local $\delta$-invariants $\delta(a,x)$, see \cite[\S2.6]{GS}. Let $(\Ah\times X)_{\delta}$ be the level set of this function. By \cite[Proposition 2.6.3]{GS}, for each $\delta_{0}\geq0$, as long as $\deg(D)$ is large enough, we have 
\begin{equation*}
\codim_{\Ah\times X}(\Ah\times X)_{\delta}\geq\delta+1, \textup{ for all }\delta\leq\delta_{0}.
\end{equation*}
Fix such a $D$ (depending on $\delta_{0}$) in the definition of $\MHit$. Let $(\Ah\times X)'=\bigsqcup_{\delta\leq\delta_{0}}(\Ah\times X)_{\delta}$, which is an open subscheme of $\Ah\times X$. Let $(\calB\times\Xz)'=(\Ah\times X)'\times_{(\Ah\times X)}(\calB\times\Xz)$.

We will need the notion of Property (G-2) of a correspondence, as defined in \cite[Definition A.6.1]{GS} and recalled in Appendix \ref{a:G2}. The following fact is an easy consequence of \cite[Lemma 3.1.4]{GS}

\begin{lemma}\label{l:grlike}
Any algebraic subspace $\hH'\subset\hHpar|_{(\calB\times\Xz)'}$ which is of finite type over $\hMpar$ via both projections satisfies Property (G-2) with respect to $(\calB\times\Xz)^{\rs}$, as a self-correspondence of $\hMpar|_{(\calB\times\Xz)'}$.
\end{lemma}

Using the formalism of cohomological correspondences in Appendix \ref{ss:corr}, the fundamental class of $\hH_{\tilw}$ gives a map
\begin{equation*}
[\hH_{\tilw}]_{\#}:\hfpar_{!}\Ql\to\hfpar_{!}\Ql.
\end{equation*}
in the category $\ind D^{b}_{T}(\calB\times\Xz)$ (with $T$ acting trivially on $\calB\times\Xz$).

Completely parallel to \cite[Theorem 3.3.3]{GS} and \cite[Proposition 5.2.1]{GS}, we have
\begin{prop}
The assignment $\tilw\mapsto[\hH_{\tilw}]_\#$ for $\tilw\in\tilW$ gives a left action of $\tilW$ on the restriction $\hfpar_{!}\Ql|_{(\calB\times\Xz)'}$. Moreover, this action coincides with the action constructed in \S\ref{ss:action}.
\end{prop}

\subsection{Global Picard stack}\label{ss:globalP} 
We recall some facts from \cite[\S4.8]{NgoFL}. For a point $a\in\Ah(k)$, we have a smooth commutative group scheme $J_{a}$ over $X$, called the {\em regular centralizer group scheme}. The global Picard stack $\calP_a$ is defined as the moduli stack of $J_a$-torsors on $X$. It acts on both $\MHit_{a}$ and $\Mpar_{a,x}$ (for any $x\in X(k)$). 

Because we work with the rigidified moduli spaces, it is more relevant to consider the group subscheme $J^{z}_{a}\subset J_{a}$ which fits into the exact sequence
\begin{equation}\label{defineJz}
1\to J^{z}_{a}\to J_{a}\to i_{z,*}J_{a,z}\to 1
\end{equation}
Here $J_{a,z}$ is the fiber of $J_{a}$ at $z$ and $i_{z}:\{z\}\to X$ is the inclusion. Let $\hP_{a}$ be the Picard stack of $J^{z}_{a}$-torsors over $X$. One may also view $\hP_{a}$ as classifying a $J_{a}$-torsor over $X$ together with a trivialization at $z$. Similar argument as in Lemma \ref{l:sch} shows that $\hP_{a}$ is in fact a group scheme, locally of finite type and smooth over $k$. The exact sequence \eqref{defineJz} gives a homomorphism of group schemes $J_{a,z}\to\hP_{a}$, and an isomorphism of Picard stacks
\begin{equation*}
\calP_{a}\cong[\hP_{a}/J_{a,z}].
\end{equation*}
As $a$ varies in $\calA^{z}$, $\{\hP_{a}\}$ form a group scheme $\hP_{\calA^{z}}$ over $\calA^{z}$. Let
\begin{equation*}
\hP=\calB\times_{\calA^{z}}\hP_{\calA^{z}}.
\end{equation*}
For $\hata=(a,t_{z})\in\calB$, the choice of $t_{z}$ gives an isomorphism $J_{a,z}\isom T$. Therefore we have an isomorphism of Picard stacks over $\calB$:
\begin{equation*}
\calB\times_{\calA}\calP\cong[\hP/T]
\end{equation*}
The group scheme $\hP$ acts on both $\hMHit$ and its parahoric variants $\hM_{\bP}$ over $\calB$.

Let $J^\flat_a$ be the finite-type N\'eron model of $J_a$ over $X$ (see \cite[\S4.8]{NgoFL}), then there is an exact sequence
\begin{equation*}
1\to J^{z}_a\to J^\flat_a\to J_{a,z}\times\prod_{x\in\Sing(a)}R_{a,x}\to1.
\end{equation*}
Here $R_{a,x}$ is an affine group scheme of finite type over $\Spec k=\Spec k(x)$, and $\Sing(a)\subset X$ is the locus where $a(x)\notin\frc^{\rs}_{D}$. Let $\calP^{\flat}_{a}$ be the Picard stack of $J^{\flat}_{a}$-torsors on $X$, and let $P^{\flat}_{a}$ be its coarse moduli space. From the above sequence we deduce an exact sequence
\begin{equation}\label{exactP}
1\to\cohog{0}{X,J^\flat_a}\to J_{a,z}\times \prod_{x\in\Sing(a)}R_{a,x}\to \hP_a\to P^\flat_a\to1.
\end{equation}

\subsection{Definition of the map $\sigma$ in \eqref{globalcentpi}}\label{ss:defineglobalsigma}
Recall from \cite[Definition 2.2.2]{GS} %ref checked 
we have the {\em universal cameral cover} defined by the Cartesian diagram
\begin{equation*}
\xymatrix{\tcA^{\Hit}\ar[r]^{\ev}\ar[d]^{q} & \frt_D\ar[d]^{q_{\frt}}\\
\AHit\times X\ar[r]^{\ev} & \frc_D}
\end{equation*}
For $a\in\AHit(k)$, the preimage $X_a:=q^{-1}(\{a\}\times X)$ is the {\em cameral curve} of $a$. Let $\tcArs\subset\tcA^{\Hit}$ (resp. $(\Ah\times X)^{\rs}\subset \AHit\times X$) be the preimage of $\frt^{\rs}_{D}$ (resp. $\frc^{\rs}_{D}$). Then $q^{\rs}:\tcArs\to(\Ah\times X)^{\rs}$ is a $W$-torsor.

Recall from \cite[paragraph before Lemma 3.2.6]{GS} that we have a canonical morphism:  % ref checked
\begin{equation}\label{section}
s:\xcoch(T)\times\tcArs\to\Grass_J\to\calP.
\end{equation}
This gives a push-forward map on homology complexes
\begin{equation*}
s_*:\Ql[\xcoch(T)]\otimes\homo{*}{\tcArs/\Ah}\to\homo{*}{\calP/\Ah}
\end{equation*}
which is $W$-invariant ($W$ acts diagonally on the two factors on the LHS and acts trivially on the RHS). Therefore, it factors through the $W$-coinvariants of $\Ql[\xcoch(T)]\otimes\homo{*}{\tcArs/\calA}$. In particular, if we restrict to $\cent$, the map $s_*$ factors through a map
\begin{equation}\label{eq:th2}
s'_{*}:\cent\otimes\homo{*}{\tcArs/\Ah}_W\to\homo{*}{\calP/\Ah}
\end{equation}
Since $q^{\rs}$ is a $W$-torsor, we have $\homo{*}{\tcArs/\Ah}_W=\homo{*}{(\Ah\times X)^{\rs}/\Ah}$. Since $(\Ah\times X)^{\rs}\to\Ah$ has connected fibers, we get 
\begin{equation}\label{Wcoinv}
\homo{0}{\tcArs/\Ah}_{W}=\homo{0}{(\Ah\times X)^{\rs}/\Ah}=\const{\Ah}.
\end{equation}
On the other hand,
\begin{equation}\label{homo0}
\homo{0}{\calP/\Ah}=\Ql[\pi_{0}(\calP/\Ah)].
\end{equation}
Therefore, the degree zero part of \eqref{eq:th2} gives the desired map $\sigma$ in \eqref{globalcentpi}.

\begin{remark}\label{r:B}
We define 
\begin{equation*}
\tcB=\tcA\times_{\AHit}\calB, \hspace{1 cm} \tcBrs=\tcArs\times_{\Ah}\calB.
\end{equation*}
When we work with the rigidified versions of Hitchin moduli spaces, the above discussion is valid if we change $\Ah,\tcArs$ and $\calP$ to $\calB,\tcBrs$ and $\hP$ respectively. In particular, we have a map
\begin{equation}\label{centB}
\sigma_{\calB}:\cent\to\Ql[\pi_{0}(\hP/\calB)].
\end{equation}
analogous to \eqref{globalcentpi}. Since $\hP\to\calB\times_{\Ah}\calP$ is a $T$-torsor, the sheaf $\pi_{0}(\hP/\calB)$ is the pullback of $\pi_{0}(\calP/\Ah)$ to $\calB$. 
\end{remark}

% proof of global

\section{Proof of the global main theorem}

\subsection{Proof of Theorem \ref{th:Hit}}\label{ss:proofHit}
We first set up some notation. Fix $S\subset T$ to be any algebraic subgroup. Then $T$, hence $S$ acts on $\hMHit$. The complex $\hfH_{!}\Ql$ (as an ind-object of $D^{b}(\calB)$) carries a canonical $S$-equivariant structure, and can be viewed as an ind-object in the $S$-equivariant derived category $D^{b}_{S}(\calB)=D^{b}([\calB/S])$, where $S$ acts trivially on $\calB$. Let $p_{S}:[\calB/S]\to\calB$ be the projection. Then we have the derived functor $p_{S,*}:D^{b}_{S}(\calB)\to D^{+}(\calB)$ of taking $S$-equivariant cohomology. We define
\begin{equation*}
\bR^{i}_{S}\hfH_{!}\Ql:=\bR^{i}p_{S,*}\hfH_{!}\Ql.
\end{equation*}
For each point $\hata\in\calB(k)$, the stalk of $\bR^{i}_{S}\hfH_{!}\Ql$ at $\hata$ is
\begin{equation*}
(\bR^{i}_{S}\hfH_{!}\Ql)_{\hata}=\cohoc{i}{[\hMHit_{\hata}/S]}.
\end{equation*}
In particular,
\begin{equation*}
(\bR^{i}_{T}\hfH_{!}\Ql)_{\hata}=\cohoc{i}{\MHit_{a}}.
\end{equation*}
Similarly, one defines $\bR^{i}_{S}\hfpar_{!}\Ql$.

We also define the following maps
\begin{equation*}
\xymatrix{\tcBrs\ar[r]^{q^{\rs}}\ar@/_1pc/[rr]_{\tilp^{\rs}} & \calB\times\Xz\ar[r]^{p} & \calB}
\end{equation*} 

We can identify $\hfH_{!}\Ql\boxtimes\const{\Xz}$ with $p^{!}\hfH_{!}\Ql[-2](-1)$. Therefore, the action of $\cent$ on $\hfH_{!}\Ql\boxtimes\const{\Xz}$ defined in \S\ref{ss:action} may be viewed as a map
\begin{equation*}
\cent\otimes p^{!}\hfH_{!}\Ql\to p^{!}\hfH_{!}\Ql.
\end{equation*}
Applying the adjunction $(p_{!},p^{!})$, we get
\begin{equation}\label{padj}
\cent\otimes\homog{*}{\Xz}\otimes\hfH_{!}\Ql=\cent\otimes p_{!}p^{!}\hfH_{!}\Ql\to\hfH_{!}\Ql.
\end{equation}
Taking only the $\homog{0}{\Xz}$ part, and passing to the level of $S$-equivariant cohomology sheaves, we get
\begin{equation}\label{definealpha}
\alpha_{0}:\cent\otimes\bR^{i}_{S}\hfH_{!}\Ql\to\bR^{i}_{S}\hfH_{!}\Ql.
\end{equation}
This gives an action of $\cent$ on $\bR^{i}_{S}\hfH_{!}\Ql$. It is easy to see that the $\cent$-action on $\bR^{i}_{S}(\hfH_{!}\Ql\boxtimes\const{\Xz})=\bR^{i}_{S}\hfH_{!}\Ql\boxtimes\const{\Xz}$, given by taking $\bR^{i}_{S}(-)$ of the construction in \S\ref{ss:action}, is the same as $\alpha_{0}\otimes\id$.

The map \eqref{section} has a rigidified version
\begin{equation*}
s:\xcoch(T)\times\tcBrs\to\Grass_J|_{\Xz}\to\hP,
\end{equation*}
which gives a geometric action of $\lambda$ on $\hMparrs$ via
\begin{equation}\label{actionlam}
\xymatrix{\hMparrs\ar@{=}[d]\ar[r]^{s(\lambda,-)} & \hP\times_{\calB}\hMparrs\ar[r]^{\act}\ar@{=}[d] & \hMparrs\ar@{=}[d]\\
\tcBrs\times_{\calB}\hMHit\ar[r] & \tcBrs\times_{\calB}(\hP\times_{\calB}\hMHit)\ar[r]^(.6){\id\times\act} & \tcBrs\times_{\calB}\hMHit}
\end{equation}

Let $\Av(\lambda):=\sum_{\lambda'\in|\lambda|}\lambda'\in\cent$, then $\Av(\lambda)$ acts on $p^{!}\hfH_{!}\Ql$. Base change from $\calB\times\Xz$ to $\tcBrs$, the action of $\Av(\lambda)$ on $\tilp^{\rs,!}\hfH_{!}\Ql=q^{\rs,!}\hfpar_{!}\Ql$ comes from the action of $\sum_{\lambda'\in|\lambda|}\lambda'$ on $\hfpar_{!}\Ql$, hence is induced from the horizontal arrows in \eqref{actionlam} for various $\lambda'\in|\lambda|$:
\begin{equation}\label{tilpull}
q^{\rs,!}\Av(\lambda):\tilp^{\rs,!}\hfH_{!}\Ql\xrightarrow{\sum_{\lambda'\in|\lambda|}s_{\lambda',*}}\tilp^{\rs,!}(\homo{*}{\hP/\calB}\otimes\hfH_{!}\Ql)\xrightarrow{\tilp^{\rs,!}(\cap)}\tilp^{\rs,!}\hfH_{!}\Ql
\end{equation}
where the last arrow $\cap$ is the cap product
\begin{equation*}
\cap:\homo{*}{\hP/\calB}\otimes\hfH_{!}\Ql\to\hfH_{!}\Ql.
\end{equation*}
derived from the action of $\hP$ on $\hMHit$. See \cite[\S7.4.2]{NgoFL} for a discussion of cap products, which extends to the situation where the relevant spaces are locally of finite type as formulated in Appendix \ref{a:lft}.

Applying the adjunction $(\tilp^{\rs}_{!},\tilp^{\rs,!})$ to \eqref{tilpull} we get
\begin{equation}\label{tilpadj}
\homo{*}{\tcBrs/\calB}\otimes\hfH_{!}\Ql=\tilp^{\rs}_{!}\tilp^{\rs,!}\hfH_{!}\Ql\xrightarrow{\sum_{\lambda'\in|\lambda|}s_{\lambda',*}}\homo{*}{\hP/\calB}\otimes\hfH_{!}\Ql\xrightarrow{\cap}\hfH_{!}\Ql
\end{equation}

On the other hand, since the map \eqref{tilpull} is $q^{\rs,!}\Av(\lambda)$, the map \eqref{tilpadj} should factor through the map \eqref{padj} via 
\begin{equation*}
q^{\rs}_{!}\otimes\id:\homo{*}{\tcBrs/\calB}\otimes\hfH_{!}\Ql\to\homo{*}{\calB\times\Xz/\calB}\otimes\hfH_{!}\Ql=\homog{*}{\Xz}\otimes\hfH_{!}\Ql.
\end{equation*}

Summarizing, as $\lambda\in\xcoch(T)$ varies, we get a commutative diagram
\begin{equation}\label{commm}
\xymatrix{\cent\otimes\homo{*}{\tcBrs/\calB}_{W}\otimes\hfH_{!}\Ql\ar[r]^(.65){s'_{*}\otimes\id}\ar[d]^{q^{\rs}_{!}} & \homo{*}{\hP/\calB}\otimes\hfH_{!}\Ql\ar[d]^{\cap}\\
\cent\otimes\homog{*}{\Xz}\otimes\hfH_{!}\Ql\ar[r]^(.6){\eqref{padj}} & \hfH_{!}\Ql}
\end{equation}
Here we take $W$-coinvariants on $\homo{*}{\tcBrs/\calB}$ because both maps from it are $W$-invariant. So the $s'_{*}$ in the upper arrow is the analog of the map \eqref{eq:th2} in the rigidified setting. 

We are now ready to pass to the $S$-equivariant cohomology sheaves. We can apply the truncation functor $\tau_{\leq0}$ (in the usual $t$-structure of $D^{b}(\calB)$) to $\homo{*}{\tcBrs/\calB}$, $\homog{*}{\Xz}\otimes\const{\calB}$ and $\homog{*}{\hP/\calB}$, and the truncation functor $\tau_{\leq i}p_{S,*}$ (in the usual $t$-structure of $D^{+}(\calB)$, recall $p_{S,*}:D^{b}_{S}(\calB)\to D^{+}(\calB)$) to the complex $\hfH_{!}\!\Ql$. The ``truncated'' diagram  \eqref{commm} is still commutative.

Taking the zeroth homology sheaves of $\homo{*}{\tcBrs/\calB}$, $\homog{*}{\Xz}$ and $\homog{*}{\hP/\calB}$, using the analogs of \eqref{Wcoinv} and \eqref{homo0} for $\calB,\tcBrs$ and $\hP$ (see Remark \ref{r:B}), and passing to $\bR^{i}_{S}\hfH_{!}\Ql$ in the truncated diagram \eqref{commm}, we get the following commutative diagram for each $i\in\ZZ$:
\begin{equation}\label{d:effcoh}
\xymatrix{\cent\otimes\bR^i_{S}\hfH_{!}\Ql\ar[r]^{\sigma_{\calB}\otimes\id}\ar[dr]^{\alpha_{0}} & \Ql[\pi_{0}(\hP/\calB)]\otimes\bR^i_{S}\hfH_{!}\Ql\ar[d]^{\cap}\\
& \bR^{i}_{S}\hfH_{!}\Ql}
\end{equation}
where the diagonal map $\alpha_{0}$ is the same as the map \eqref{definealpha}. 

For the original statement of Theorem \ref{th:Hit}, we take $S=T$. Let $\pi_{z}:\calB\times\Xz\to\calA^{z}\times\Xz$. Since
\begin{equation*}
\hMHit\to(\calB\times\Xz)\times_{(\Ah\times X)}\MHit
\end{equation*}
is a $T$-torsor, we have
\begin{equation*}
\bR^{i}_{T}\hfH_{!}\Ql\boxtimes\const{\Xz}=\pi^{*}_{z}\bR^{i}\fHQl|_{\calA^{z}}\boxtimes\const{\Xz}.
\end{equation*}
Moreover, since $\pi_{z}$ is finite \'etale, $\End(\bR^{i}\fHQl|_{\calA^{z}}\boxtimes\const{\Xz})\to\End(\bR^{i}_{T}\hfH_{!}\Ql\boxtimes\const{\Xz})$ is injective. Therefore the diagram \eqref{d:effcoh} implies that the action of $\cent$ on $\bR^{i}\fHQl|_{\calA^{z}}\boxtimes\const{\Xz}$ factors through $\pi_{0}(\calP/\calA^{z})$. Since $\Ah\times X$ is covered by $\calA^{z}\times\Xz$ for various $z\in X(k)$, the theorem follows.

% comparison for parabolic Hitchin

\subsection{Hecke modification at ``another'' point}
This subsection provides preparatory tools for proving Theorem \ref{th:mainpar}.

Consider the correspondence
\begin{equation}\label{d:twoptmod}
\xymatrix{ & \hHk_{2} \ar[dl]_{\overleftarrow{h_{2}}}\ar[dr]^{\overrightarrow{h_{2}}} & \\
\hMpar\times \Xz\ar[r]_{\hfpar\times \id} & \calB\times(\Xz)^2 & \hMpar\times \Xz\ar[l]^{\hfpar\times \id}}
\end{equation} 
For any scheme $S$, $\hHk_{2}(S)$ is the isomorphism classes of tuples $(x,y,\calE_1,\varphi_1,t_{1,z},\iota_{1,z},\calE^B_{1,x},\calE_2,\varphi_2,t_{2,z},\iota_{2,z},\calE^B_{2,x},\tau)$ where
\begin{itemize}
\item $(x,\calE_i,\varphi_i,t_{i,z},\iota_{i,z},\calE^B_{i,x})\in\hMpar(S)$, for $i=1,2$;
\item $y\in X(S)$ with graph $\Gamma(y)$;
\item $\tau$ is an isomorphism of objects on $S\times\Xz-\Gamma(y)$:
\begin{equation*}
\tau:(\calE_1,\varphi_1,t_{1,z},\iota_{1,z})|_{S\times \Xz-\Gamma(y)}\isom(\calE_2,\varphi_2,t_{2,z},\iota_{2,z})|_{S\times \Xz-\Gamma(y)}.
\end{equation*}
\end{itemize}

For a point $(\hata,x,y)\in(\calB\times(\Xz)^2)(k)$ such that $x\neq y$, the fibers of $\overleftarrow{h_{2}}$ and $\overrightarrow{h_{2}}$ over $(\hata,x,y)$ are isomorphic to the product of $\Spr_{\bG,\gamma_{a,y}}$ and a Springer fiber in $G/B$ corresponding to the image of $\gamma_{a,x}$ in $\frg$ (here $\gamma_{a,x}\in\frg(\calO_{x})$ and $\gamma_{a,y}\in\frg(\calO_{y})$ are Kostant sections of $a$ in the formal neighborhood of $x$ and $y$; see the discussion in \S\ref{ss:prod}). If we restrict to the diagonal $\Delta_{\Xz}:\calB\times \Xz\subset\calB\times (\Xz)^2$, $\hHk_{2}|_{\Delta_{\Xz}}$ is the same as $\hHpar$. The reader may notice the analogy between our situation and the situation considered by Gaitsgory in \cite{Ga}, where he uses Hecke modifications at two points to deform the product $\Grass_G\times G/B$ to $\Fl_G$.

As in the case of $\hHpar$, we have a morphism
\begin{equation*}
\hHk_{2}\to\hMpar\times_{(\calB\times\Xz)}\hMpar\to\tcB\times_{(\calB\times\Xz)}\tcB.
\end{equation*}
Let $\hHk_{2,[e]}$ be the preimage of the diagonal $\tcB\subset\tcB\times_{\calB\times\Xz}\tcB$. One the other hand, we have the Hecke correspondence $\hHk^{\Hit}$ of $\hMHit\times\Xz$ which modifies the Hitchin pair at one point. We have a commutative diagram of correspondences
\begin{equation}\label{d:twoptbch}
\xymatrix{\hHk_{2,[e]}\ar@<-1ex>@/_/[d]\ar@<1ex>@/^/[d]\ar[r]^{\tilp_{\hH}} & \hHk^{\Hit}\ar@<-1ex>@/_/[d]\ar@<1ex>@/^/[d]\\
\hMpar\times \Xz\ar[d]^{\hfpar\times\Xz}\ar[r] & \hMHit\times \Xz\ar[d]^{\hfH\times\id}\\
\tcB\times \Xz\ar[r]^{\tilp\times\id} & \calB\times \Xz}
\end{equation}
where the horizontal maps are given by forgetting the $B$-reductions.

By \cite[Lemma 2.6.1]{GS}, if we restrict the base spaces (the bottom line above) to $\tcBrs\times\Xz\to(\calB\times\Xz)^{\rs}$, all squares in the above diagram are Cartesian. Recall from \cite[Construction 6.6.3]{GS} that for each $W$-orbit $|\lambda|$ in $\xcoch(T)$, we have a graph-like closed substack $\calH_{|\lambda|}\subset\Hecke^{\Hit}$. Similarly, we have $\hH_{|\lambda|}\subset\hHk^{\Hit}$. Let $\hH_{2,|\lambda|}\subset\hHk_{2,[e]}$ be closure of $\tilp^{-1}_{\hH}\hH^{\rs}_{|\lambda|}$. Using the formalism of cohomological correspondences in Appendix \ref{ss:corr}, the fundamental class $[\hH_{2,|\lambda|}]$ induces a map
\begin{equation*}
[\hH_{2,|\lambda|}]_{\#}:\hfpar_{!}\Ql\boxtimes\const{\Xz}\to\hfpar_{!}\Ql\boxtimes\const{\Xz}.
\end{equation*}

Parallel to \cite[Proposition 6.6.4]{GS}, we have
\begin{lemma}\label{l:twoptaction}
The assignment $\cent\ni\Av(\lambda)\mapsto[\hH_{2,|\lambda|}]_{\#}$ extends by linearity to an action of $\cent$ on $(\hfpar_{!}\Ql\boxtimes\const{\Xz})|_{(\calB\times(\Xz)^{2})'}$, where $(\calB\times(\Xz)^{2})'=(\calB\times\Xz)'\times_{\calB}(\calB\times\Xz)'$.
\end{lemma}

Let $p_2:\calB\times(\Xz)^2\to\calB$ be the projection. Similar argument as in \S\ref{ss:proofHit} proves the following statement. Note that the argument in \S\ref{ss:proofHit} did not use more information than the explicit knowledge of the $\cent$-action over the locus $(\calB\times\Xz)^{\rs}$.

\begin{lemma}\label{l:twofactor}
The action of $\cent$ on $\bR^{i}_{S}\hfpar_{!}\Ql\boxtimes\const{\Xz}|_{(\calB\times(\Xz)^{2})'}$ constructed in Lemma \ref{l:twoptaction} factors through the action of $p_{2}^*\Ql[\pi_0(\hP/\calB)]$ (which is also the pull-back of $\Ql[\pi_{0}(\calP/\Ah)]$) on the first factor of $\bR^{i}_{S}\hfpar_{!}\Ql\boxtimes\const{\Xz}|_{(\calB\times(\Xz)^{2})'}$ via the homomorphism \eqref{centB}.
\end{lemma}

% proof global

\subsection{Proof of Theorem \ref{th:mainpar}}\label{ss:proofmain}
We consider two actions of $\cent$ on $\hfpar_{!}\Ql|_{(\calB\times\Xz)'}$:
\begin{itemize}
\item The action $\alpha_{1}$ of $\cent$ on $\hfpar_{!}\Ql$ is given by the restriction of the $\tilW$-action constructed in \S\ref{ss:action}. This is the action involved in the statement of Theorem \ref{th:mainpar}.

\item The action $\alpha_{\Delta}$ of $\cent$ on $\hfpar_{!}\Ql$ is given by restricting the action constructed in Lemma \ref{l:twoptaction} via $\Delta_{\Xz}:(\calB\times\Xz)'\hookrightarrow(\calB\times (\Xz)^2)'$. Note that $\hfpar_{!}\Ql=\Delta_{\Xz}^*(\hfpar_{!}\Ql\boxtimes\const{\Xz})$.
\end{itemize}

We claim that the actions $\alpha_{1}$ and $\alpha_\Delta$ are the same. On the one hand, by Lemma \ref{l:twoptaction}, the action of $\Av(\lambda)$ on $\hfpar_{!}\Ql\boxtimes\const{\Xz}|_{(\calB\times(\Xz)^{2})'}$ is given by the class $[\hH_{2,|\lambda|}]$. Recall from \cite[Appendix A.2]{GS} the notion of the pullback of a cohomological correspondence. By \cite[Lemma A.2.1]{GS}, the action of $\alpha_\Delta(\Av(\lambda))$ is given by the pullback class $\Delta_{\Xz}^*[\hH_{2,|\lambda|}]\in\Corr(\hH';\Ql,\Ql)$, where $\hH'\subset\hHpar$ is a large enough closed sub-correspondence with both maps to $\hMpar$ proper. For the notation $\Corr(-;-,-)$, see Appendix \ref{ss:corr}. Let $(\calB\times(\Xz)^{2})^{\rs}=(\calB\times\Xz)^{\rs}\times_{\calB}(\calB\times\Xz)^{\rs}$, and $\Delta^{\rs}_{\Xz}:(\calB\times\Xz)^{\rs}\hookrightarrow(\calB\times(\Xz)^{2})^{\rs}$ be the restriction of $\Delta_{\Xz}$. Over $(\calB\times(\Xz)^{2})'$, $\hH^{\rs}_{2,|\lambda|}$ is finite \'etale over $\hMpar\times\Xz$ via both projections, therefore $\Delta_{\Xz}^{\rs,*}[\hH^{\rs}_{2,|\lambda|}]=[\hH^{\rs}_{2,|\lambda|}|_{\Delta^{\rs}_{\Xz}}]$, the fundamental class of the base change of $\hH^{\rs}_{2,|\lambda|}$ via $\Delta^{\rs}_{\Xz}$. However, restricting the left column of the diagram \eqref{d:twoptbch} to $\tcBrs\hookrightarrow\tcB\times\Xz$ (as the graph of the projection $\tcBrs\to\Xz$), we get the following Cartesian diagram
\begin{equation*}
\xymatrix{\bigsqcup_{\lambda'\in|\lambda|}\hH^{\rs}_{\lambda'}\ar@<-1ex>@/_/[d]\ar@<1ex>@/^/[d]\ar[r]  & \hH^{\rs}_{|\lambda|}\ar@<-1ex>@/_/[d]\ar@<1ex>@/^/[d]\\
\hMparrs\ar[d]^{\hfpar}\ar[r] & \hMHit\times\Xz|_{(\calB\times\Xz)^{\rs}}\ar[d]^{\hfH\times\id}\\
\tcBrs\ar[r]^{q^{\rs}} & (\calB\times\Xz)^{\rs}}
\end{equation*}
Hence
\begin{equation}\label{twoclassrs}
\Delta_{\Xz}^{\rs,-1}(\hH^{\rs}_{2,|\lambda|})=\bigsqcup_{\lambda'\in|\lambda|}\hH^{\rs}_{\lambda'}.
\end{equation}

On the other hand, the action of $\alpha_{1}(\Av(\lambda))$ is given by the class $\sum_{\lambda'\in|\lambda|}[\hH_{\lambda'}]\in\Corr(\hH';\Ql,\Ql)$ (enlarge the correspondence $\hH'$ defined before to contain all the $\hH_{\lambda'}$ with $\lambda'\in|\lambda|$). Over $(\calB\times\Xz)'$, both classes $\Delta^{*}_{\Xz}[\hH_{2,|\lambda|}]$ and $\sum_{\lambda'\in|\lambda|}[\hH_{\lambda'}]$ are supported on $\hH'$, which has property (G-2) by Lemma \ref{l:grlike}. Applying Lemma \ref{l:openpart}, the equality \eqref{twoclassrs} implies
\begin{equation*}
(\Delta^{*}_{\Xz}[\hH_{2,|\lambda|}])_{\#}=\sum_{\lambda'\in|\lambda|}[\hH_{\lambda'}]_{\#}
\end{equation*}
as endomorphisms of $\hfpar_{!}\Ql|_{(\calB\times\Xz)'}$. This proves the claim.

By Lemma \ref{l:twofactor}, the action $\alpha_\Delta$ of $\cent$ on $\bR^{i}_{S}\hfpar_{!}\Ql|_{(\calB\times\Xz)'}=\Delta_{\Xz}^{*}(\bR^i_{S}\hfpar_{!}\Ql\boxtimes\const{\Xz}|_{(\calB\times(\Xz)^{2})'})$ factors through \eqref{centB}. Since the action of $\alpha$ on $\bR^{i}_{S}\hfpar_{!}\Ql|_{(\calB\times\Xz)'}$ is the same as $\alpha_{\Delta}$, it also factors through $\eqref{centB}$.

Again, for the purpose of proving Theorem \ref{th:mainpar}, we take $S=T$ and argue as in the final part of \S\ref{ss:proofHit}. Taking the stalk at a point of $(\calB\times\Xz)'$, we conclude:

\begin{cor}\label{c:mainstalk}
Let $S\subset T$ be an algebraic subgroup and $(\hata,x)\in(\calB\times\Xz)'(k)$. Then for each $i$ the action of $\cent$ on $\cohoc{i}{[\hMpar_{\hata,x}/S]}$ factors through the action of $\pi_{0}(\calP_{a})$ via $\sigma_{a}$. 
\end{cor}

% global to local

\section{From global to local}
In this section, we prove the local main theorem (Theorem \ref{th:mainloc}) and Theorem \ref{th:mainhomo}.

\subsection{The local data} Without loss of generality, we may assume $G$ is of adjoint type.

We use the notations from \S\ref{ss:definesigma}. Let $\gamma\in\frg(F)$ be a regular semisimple element. Assume $a(\gamma)\in\frc(\calO_{F})$, otherwise the affine Springer fiber is empty. Recall the local cameral cover diagram \eqref{localcameral}. As before, we choose a component $\Spec\tilO^{!}_{F}$ of $\Spec\tilO_{F}$ with stabilizer $W^{!}$. By our assumption on char$(k)$, $F^{!}=\Frac(\tilO^{!}_{F})$ is a tamely ramified extension of $F$, hence $W^{!}=\Gal(F^{!}/F)$ is a cyclic subgroup of $W$. Fix a generator $w$ of $W^{!}$ and a primitive $m$-th root of unity $\zeta\in\mu_{m}(k)$, where $m=\#W^{!}$. 

\subsection{The global data}
Let $X=\PP^1$. Fix $z\in X(k)\backslash\{0,\infty\}$. Let $\tilX$ be the contracted product $\tilX=W\twtimes{W^{!}}\PP^1$, where $w$ acts on $\PP^1$ by $t\mapsto\zeta t$. We take $\tilX^{!}$ to be the component $\{1\}\times\PP^{1}\subset\tilX$. The morphism $\pi:\tilX\to\tilX^{!}\sslash W^{!}=X$ is a branched $W$-cover. The point $0$ has a unique preimage in $\tilX^{!}$, which we denote by $0^{!}$. We fix a $W^{!}$-equivariant isomorphism
\begin{equation*}
\tilO^{!}_{F}\cong\calO_{\tilX^{!},0^{!}}
\end{equation*}
This induces a $W$-equivariant isomorphism $\tilO_{F}\cong\calO_{\tilX,0}$ (the latter being the completion of $\tilX$ along $\pi^{-1}(0)$), and also an isomorphism $\calO_{F}\cong\calO_{X,0}$. 

\begin{lemma}\label{l:appr}
For any $N>0$, there exists an integer $d=d(N)$ such that for any divisor $D$ on $X$, disjoint from $\{0,z,\infty\}$ and $\deg(D)\geq d$, there exists a section $a:X\to\frc_{D}$ such that
\begin{enumerate}
\item There is a $W$-equivariant birational morphism $\tilX\to X_{a}$ (i.e., $\tilX$ is the normalization of the cameral curve $X_{a}$);
\item Let $a_{0}$ be the restriction of $a$ to $\Spec\calO_{X,0}$. Then $a_{0}\equiv a(\gamma)\mod\varpi^{N}$ (since $D$ is disjoint from $0$, $a_{0}$ is an element in $\frc(\calO_{X,0})=\frc(\calO_{F})$);
\item $a(z)\in\frc^{\rs}$, i.e., $a\in\calA^{z}$.
\end{enumerate}
\end{lemma}
\begin{proof}
To give such a section $a$ is equivalent to giving a $W$-equivariant morphism $\tila:\tilX\to\frt_{D}$ whose induced map between $W$-quotients satisfies local conditions (2) and (3) (which is then automatically birational because $a$ is regular semisimple at $z$).

Let $\calL$ be the following coherent sheaf on $X$
\begin{equation*}
\calL=(\frt\otimes\pi_{*}\calO_{\tilX})^{W}.
\end{equation*}
Since char$(k)$ does not divide $\#W$, this is a vector bundle over $X$. First fix any divisor $D$ on $X$ not containing $\{0,z,\infty\}$. Giving a $W$-equivariant map $\tilX\to\frt_{D}$ is equivalent to giving a section $\tila\in\cohog{0}{X,\calL(D)}$. The conditions (2)(3) will be satisfied if we can find such a section $\tila$ which approximates given sections at $0$ and $z$ to high order. This can obviously be achieved provided $\deg(D)$ is large.
\end{proof}

Let $N$ be the integer in Proposition \ref{p:constancy}: whenever $a(\gamma')\equiv a(\gamma)\mod\varpi^{N}$, then Theorem \ref{th:mainloc} is true for $a(\gamma')$ if and only if it is true for $a(\gamma)$. Applying Lemma \ref{l:appr} to this $N$, we obtain a divisor $D=2D'$ on $X$ and $a:X\to\frc_{D}$ satisfying the properties therein. By choosing $\deg(D)$ large enough, we can make sure that $(\Ah\times X)'\supset(\Ah\times X)_{\delta}$ for $\delta=\delta(a;0)$, hence $(a,0)\in(\Ah\times X)'$ (see the discussion in \S\ref{ss:BXprime}). We use this $D$ to define the rigidified parabolic Hitchin moduli $\hMpar$. With this $a$ we have the cameral curve $q_{a}:X_{a}\to X$. The image of $\tilX^{!}$ in $X_{a}$ is a component $X^{!}_{a}$. Let $\Sing(a)\subset X$ be the locus where $q_{a}:X_{a}\to X$ is not \'etale, together with the point $\{0\}$.

\subsection{Relation between local and global Picard}\label{ss:localP}
Let $x\in X(k)$. The discussion in \S\ref{ss:localPic} can be applied to $\gamma(a,x)\in\frg(F_{x})$, the Kostant section of the restriction of $a$ to the formal disk $\Spec\calO_{x}$. In particular, we obtain $\calP_{a,x}, P_{a,x}, \Lambda_{a,x}$ and $R_{a,x}$, and an exact sequence
\begin{equation}\label{newlocalP}
1\to R_{a,x}\to P_{a,x}\to\Lambda_{a,x}\to1.
\end{equation}
Note that $R_{a,x}$ is the same group scheme which appeared in \eqref{exactP}. For each $x\in X(k)\backslash\{z\}$, we have a commutative diagram
\begin{equation}\label{PaxP}
\xymatrix{P_{a,x}\ar[r]\ar[d] & \hP_a\ar[d]\\ \Lambda_{a,x}\ar[r] & P^{\flat}_a}
\end{equation}

Let $\hata=(a,t_{z})\in\calB(k)$ lifting $a$ such that $t_{z}\in X^{!}_{a}$. The choice of $t_{z}$ gives an identification $J_{a,z}\isom T$. Recall the exact sequence \eqref{exactP}. By \cite[\S 4.11]{NgoFL} we have $\cohog{0}{X,J^{\flat}_{a}}=T^{w}$, hence the first arrow in \eqref{exactP} gives a map $T^{w}\to J_{a,z}=T$, which is easily checked to be the canonical inclusion $T^{w}\hookrightarrow T$. Let $S\subset T$ be a torus which is complementary to the neutral component of $T^{w}$, so that $S\cap T^{w}$ is a finite diagonalizable group over $k$. By the assumption on char$(k)$, $S\cap T^{w}$ is in fact discrete.

Combining the exact sequences \eqref{exactP}, \eqref{newlocalP} and the diagram \eqref{PaxP}, we get a map between exact sequences
\begin{equation}\label{prodP}
\xymatrix{1\ar[r] & S\times\prod_{x\in\Sing(a)}R_{a,x}\ar[r]\ar[d]^{\alpha} & S\times\prod_{x\in\Sing(a)}P_{a,x}\ar[r]\ar[d]^{\beta} & \prod_{x\in\Sing(a)}\Lambda_{a,x}\ar[d]^{\gamma}\ar[r] & 1\\
1\ar[r] & (J_{a,z}\times\prod_{x\in\Sing(a)}R_{a,x})/T^{w}\ar[r] & \hP_{a}\ar[r] & P^{\flat}_{a}\ar[r] & 1}
\end{equation}
By the choice of $S$, $\alpha$ is surjective. Let $P^{\ker}_{a}=\ker(\beta)$ and $\Lambda_{a}^{\ker}=\ker(\gamma)$. The snake lemma gives an exact sequence between the kernels
\begin{equation}\label{Pker}
1\to S\cap T^{w}\to P_{a}^{\ker}\to\Lambda_{a}^{\ker}\to1.
\end{equation}
Here $S\cap T^{w}$ is viewed as a subgroup of $T^{w}=\cohog{0}{X,J^{\flat}_{a}}$, hence maps diagonally into $S\times\prod_{x}R_{a,x}$. Note that $S\cap T^{w}$ is a discrete group scheme, so is $P^{\ker}_{a}$, hence we may identify $P^{\ker}_{a}$ with its $k$-points.

By Lemma \ref{l:lattice} below, the map $\gamma$ in \eqref{prodP} is also surjective. Hence so is $\beta$ and we have an exact sequence
\begin{equation}\label{addS}
1\to P^{\ker}_{a}\to S\times\prod_{x\in\Sing(a)}P_{a,x}\to\hP_{a}\to1
\end{equation}

% computing the P^{\flat}_{a}

\begin{lemma}\label{l:lattice}      
In our situation, the coarse moduli space $P^{\flat}_{a}$ of $\calP^{\flat}_{a}$ is a discrete group. The natural map $\iota_{0}:\Lambda_{a,0}\to P^{\flat}_a$ is an isomorphism, and both groups are canonically isomorphic to $\xcoch(T)_{w}/\textup{torsion}$.
\end{lemma}
\begin{proof}
By \cite[Corollaire 4.8.1]{NgoFL}, the finite type N\'eron model $J^{\flat}_{a}$ of $J_{a}$ is $(\Res_{\tilX^{!}/X}(T\times \tilX^{!}))^{w}$. The Lie algebra of $P^{\flat}_{a}$ is $\cohog{1}{X,\Lie J^{\flat}_{a}}=\cohog{1}{X,(\frt\otimes\calO_{\tilX^{!}})^{w}}\subset\cohog{1}{\tilX^{!},\calO_{\tilX^{!}}}=0$ since $\tilX^{!}\cong\PP^{1}$. Therefore $P^{\flat}_{a}$ is discrete.

The restriction of $J^{\flat}_{a}$ to $\Spec\calO_{F}$ is $(\Res_{\tilO^{!}_{F}/\calO_{F}}T)^{w}$. Therefore $\Lambda_{a,0}=J_{a}(F)/J^{\flat}_{a}(\calO_{F})=T(F^{!})^{w}/T(\tilO^{!}_{F})^{w}$. Taking the $w$-invariants of the exact sequence
\begin{equation*}
1\to T(\tilO^{!}_{F})\to T(F^{!})\xrightarrow{\ev}\xcoch(T)\to0
\end{equation*}
we get that $\Lambda_{a,0}$ is the image of the map
\begin{equation*}
\ev:J_{a}(F)=T(F^{!})^{w}\to\xcoch(T)^{w}.
\end{equation*}
Similarly, from the injection $J^{\flat}_{a}\hookrightarrow\Res_{\tilX^{!}/X}(T\times \tilX^{!})$ we deduce a natural map
\begin{equation*}
\deg:P^{\flat}_{a}\to\cohog{1}{\tilX^{!},T}^{w}=\xcoch(T)^{w}.
\end{equation*}

Consider the diagram 
\begin{equation}\label{ww}
\xymatrix{\xcoch(T)_{w}\ar@{=}[d]\ar[r]^{\alpha} & \Lambda_{a,0}\ar[r]^{\ev}\ar[d]^{\iota} & \xcoch(T)^{w}\ar@{=}[d]\\
\xcoch(T)_{w}\ar[r]^{\beta} & P^{\flat}_{a}\ar[r]^{\deg} & \xcoch(T)^{w}}
\end{equation}
Here $\alpha$ is a surjection induced from the map $J^{\flat,0}_{a}\to J^{\flat}_{a}$ (here $(-)^{0}$ denotes fiber-wise neutral component), and the fact that $J_{a}(F)/J^{\flat,0}_{a}(\calO_{F})\cong\xcoch(T)_{w}$ \cite[Lemme 3.9.4]{NgoFL}. The map $\beta$ is a surjection induced from the same map $J^{\flat,0}_{a}\to J^{\flat}_{a}$, and the fact that $\pi_{0}(\Pic(X,J_{a}^{\flat,0}))\cong\xcoch(T)_{w}$ \cite[Proposition 6.4, Corollaire 6.7]{NgoHit}. The diagram \eqref{ww} is easily seen to be commutative.

The fact that $\beta$ is surjective shows that $\iota_{0}$ is also surjective. The fact that $\ev$ is injective shows that $\iota_{0}$ is also injective. Therefore $\iota_{0}$ is an isomorphism. Now $\ker(\alpha)\cong J^{\flat}_{a}(\calO_{F})/J^{\flat,0}_{a}(\calO_{F})$ is torsion and $\Lambda_{a,0}$ is torsion-free since $\ev$ is injective, hence the first row of \eqref{ww} identifies $\Lambda_{a,0}$ with the torsion-free quotient of $\xcoch(T)_{w}$.
\end{proof}

\begin{lemma}\label{l:inj} The natural map
\begin{equation*}
P^{\ker}_{a}\to\prod_{x\in\Sing(a)\backslash\{0\}}\pi_{0}(P_{a,x})
\end{equation*}
is injective with finite cokernel.
\end{lemma}
\begin{proof}
Consider the following map between short exact sequences
\begin{equation}\label{kerlocal}
\xymatrix{1\ar[r] & S\cap T^{w}\ar[d]^{\alpha}\ar[r] & P^{\ker}_{a}\ar[r]\ar[d]^{\beta} & \Lambda^{\ker}_{a}\ar[r]\ar[d]^{\gamma} & 1\\
1\ar[r] & \prod_{x\in\Sing(a)\backslash\{0\}}\pi_{0}(R_{a,x})\ar[r] & \prod_{x\in\Sing(a)\backslash\{0\}}\pi_{0}(P_{a,x})\ar[r] & \prod_{x\in\Sing(a)\backslash\{0\}}\Lambda_{a,x}\ar[r] & 1}
\end{equation}
By Lemma \ref{l:lattice}, $\Lambda_{a,0}\isom P^{\flat}_{a}$, hence $\gamma$ is an isomorphism. By \cite[Proposition 3.9.7]{NgoFL}, we have an isomorphism $\pi_{0}(P_{a,\infty})\cong\xcoch(T)_{w}$ because $G$ is of adjoint type. Also by \cite[Proposition 3.9.7]{NgoFL}, the map $\pi_{0}(T^{w})\to\xcoch(T)_{w}=\pi_{0}(P_{a,\infty})$ is an embedding whose image is the torsion part of $\xcoch(T)_{w}$. Since $T^{w}\to P_{a,\infty}$ factors through $R_{a,\infty}$, we see that $\pi_{0}(T^{w})\to\pi_{0}(R_{0,\infty})$ is injective. By the choice of $S$ (complementary to the neutral component of $T^{w}$), $S\cap T^{w}$ injects into $\pi_{0}(T^{w})$, therefore $S\cap T^{w}\hookrightarrow\pi_{0}(R_{a,\infty})$, proving that the map $\alpha$ in \eqref{kerlocal} is also injective. Therefore the map $\beta$ in \eqref{kerlocal} is also injective.

Applying the snake lemma to \eqref{kerlocal}, we get $\coker(\beta)\cong\coker(\alpha)$, which is finite.
\end{proof}

% compare defn of cent to pi0

\begin{prop}\label{p:comppi}  % ref checked
For any $x\in X(k)$, the following diagram is commutative:
\begin{equation}\label{comppi}
\xymatrix{\cent\ar[r]^{\sigma_{a,x}}\ar@/_2pc/[rr]^{\sigma_{a}} & \Ql[\pi_{0}(P_{a,x})]\ar[r] & \Ql[\pi_{0}(\calP_{a})]}
\end{equation}
Here $\sigma_{a,x}$ is the homomorphism \eqref{localcentpi} applied to $\gamma(a,x)$, and $\sigma_{a}$ is the stalk of \eqref{globalcentpi} at $a$.
\end{prop}
\begin{proof}
Recall the choice $t_{z}$ gives a point $z^{!}\in X^{!,\rs}_{a}$ over $z$. Also we have $x^{!}\in X^{!}_{a}$ over $x$. Consider the diagram
\begin{equation}\label{xz}
\xymatrix{& \pi_{0}(P_{a,x})\ar[dr]^{\iota_{x}} \\
\xcoch(T)\ar[dr]^{s(-,z^{!})}\ar[ur]^{\tilsigma_{a,x^{!}}}\ar[rr]^{\tau} & & \pi_{0}(\calP_{a})\\
& \pi_{0}(\Grass_{J_{a},z})=\pi_{0}(P_{a,z})\ar[ur]^{\iota_{z}}}
\end{equation}
Here $s(-,z^{!})$ is the map in \eqref{section}, and the composition $\iota_{z}\cdot s(-,z^{!})$ is used to define $\sigma_{a}$, see \S\ref{ss:defineglobalsigma}. The homomorphism $\tilsigma_{a,x^{!}}$ is defined in \eqref{xcochgamma} (quoted from \cite[Proposition 3.9.2]{NgoFL}), which was used to define $\sigma_{a,x}$. Therefore, in order to show that \eqref{comppi} is commutative, if suffices to show that the outer square of \eqref{xz} is commutative.

The arrow $\tau$ in \eqref{xz} is defined in \cite[Proposition 6.8]{NgoHit} and \cite[Proposition 4.10.3]{NgoFL}. We shall show that the two triangles in \eqref{xz} are both commutative. The commutativity of the upper triangle follows from the compatibility of Ng\^o's constructions, see the proof of \cite[Proposition 4.10.3]{NgoFL}. On the other hand, tracing through the definition of $s(-,z^{!})$ in \cite[Lemma 3.2.5]{GS}, we see that $s(-,z^{!})$ is the same as Ng\^o's map $\tilsigma_{a,z^{!}}$, therefore the lower triangle of \eqref{xz} is also commutative. This finishes the proof.
\end{proof}

% product formula

\subsection{Product formula}\label{ss:prod}
For each $x\in\Sing(a)$, choosing a trivialization of $\calO_{X}(D)$ at $y$, we may view the restriction of $a$ on $\Spec\calO_{x}$ as an element $a_{x}\in\frc(\calO_{x})$. Let $\gamma(a,x)=\epsilon(a_{x})\in\frg(\calO_{x})$ be the Kostant section. Define
\begin{equation*}
\Spr_{\bP,a,x}:=\Spr_{\bP,\gamma(a,x)}.
\end{equation*}
Recall from \cite[\S4.2.4]{NgoFL} that we have a global Kostant section $\epsilon(a)=(\calE,\varphi)\in\MHit_{a}$. Picking any isomorphism $\iota_{z}:(\calE,\varphi)_{z}\isom(G,t_{z})$ gives a point $\hep(\hata)\in\hMHit_{\hata}$. As in \cite[\S4.15]{NgoFL}, by gluing the local Hitchin pairs at $x\in\Sing(a)$ with $\hep(\hata)$ we get a morphism
\begin{equation}\label{preprod}
\Spr_{a,0}\times\prod_{x\in\Sing(a)\backslash\{0\}}\Spr_{\bG,a,x}\to\hMpar_{\hata,0}
\end{equation}
which intertwines the $\prod_{x\in\Sing(a)}P_{a,x}$-action on the LHS and the $\hP_{a}$-action on the right.

A rigidified version of the product formula (\cite[Proposition 4.15.1]{NgoFL}, \cite[Proposition 2.4.1]{GS}) gives a homeomorphism
\begin{equation}\label{prod}
(\Spr_{a,0}\times\prod_{x\in\Sing(a)\backslash\{0\}}\Spr_{\bG,a,x})\twtimes{\prod_{x\in\Sing(a)}P_{a,x}}\hP_a\to\hMpar_{\hata,0}.
\end{equation}
Dividing both sides of \eqref{prod} by $S$, using \eqref{addS}, we get a homeomorphism of stacks
\begin{equation}\label{modS}
\Spr_{a,0}\twtimes{P_{a}^{\ker}}\prod_{x\in\Sing(a)\backslash\{0\}}\Spr_{\bG,a,x}\to[\hMpar_{\hata,0}/S].
\end{equation}

\subsection{Pulling apart components}\footnote{This part of the argument was suggested by Y.Varshavsky.} By \cite[\S3.10.2]{NgoFL}, the irreducible components of $\Spr_{\bG,a,x}$ are in bijections with $\pi_{0}(P_{a,x})$. Hence, by Lemma \ref{l:inj}, $P_{a}^{\ker}$ hence permutes $\prod_{x\in\Sing(a)\backslash\{0\}}\Irr(\Spr_{\bG,a,x})$ freely with finitely many orbits. Let $Y\subset\prod_{x\in\Sing(a)\backslash\{0\}}\Spr_{\bG,a,x}$ be a union of irreducible components, representing the orbits of the $P_{a}^{\ker}$-action. Let $Y^{\reg}\subset Y$ be the intersection of $Y$ with the regular locus $\prod\Spr^{\reg}_{\bG,a,x}$ (cf. \cite[Lemma 3.3.1]{NgoFL}). Then we have a commutative diagram
\begin{equation*}
\xymatrix{\Spr_{a,0}\times Y^{\reg}\ar@{^{(}->}[r]^{j_{Y}}\ar[d]^{\wr} & \Spr_{a,0}\times Y\ar[d]^{\nu}\\
\Spr_{a,0}\twtimes{P_{a}^{\ker}}\prod_{x\in\Sing(a)\backslash\{0\}}\Spr^{\reg}_{\bG,a,x}\ar@{^{(}->}[r]^{j} & \Spr_{a,0}\twtimes{P_{a}^{\ker}}\prod_{x\in\Sing(a)\backslash\{0\}}\Spr_{\bG,a,x}\ar[r] & [\hMpar_{\hata,0}/S]}
\end{equation*} 
The map $j_{Y,!}$ on compactly supported cohomology can be factored as\begin{equation}\label{apart}
j_{Y,!}:\cohoc{*}{\Spr_{a,0}\times Y^{\reg}}\xrightarrow{j_{!}}\cohoc{*}{[\hMpar_{\hata,0}/S]}\xrightarrow{\nu^{*}}\cohoc{*}{\Spr_{a,0}\times Y}.
\end{equation}
Here $\nu^{*}$ makes sense because it is proper. Let $d=\dim Y$. Using the K\"unneth formula, and taking only the top cohomology of $Y^{\reg}$ and $Y$, we get
\begin{equation*}
\xymatrix{\cohoc{i}{\Spr_{a,0}}\otimes\cohoc{2d}{Y^{\reg}}\ar[r] & \cohoc{i+2d}{[\hMpar_{\hata,0}/S]}\ar[r] & \cohoc{i}{\Spr_{a,0}}\otimes\cohog{2d}{Y}}
\end{equation*}
The composition of the above maps is an isomorphism because $\cohoc{2d}{Y^{\reg}}\isom\cohog{2d}{Y}=\Ql[\Irr(Y)]$. Therefore the map $j_{!}$ in \eqref{apart} gives an injection
\begin{equation}\label{embed}
\cohoc{i}{\Spr_{a,0}}\hookrightarrow\cohoc{i}{\Spr_{a,0}}\otimes\cohoc{2d}{Y^{\reg}}\hookrightarrow\cohoc{i+2d}{[\hMpar_{\hata,0}/S]}.
\end{equation}
where the first map is given by the inclusion of the fundamental class $[Y^{\reg}]\in\cohoc{2d}{Y^{\reg}}$.

\begin{prop}\label{p:comparison}
The map \eqref{embed} intertwines the $\tilW\times\pi_0(P_{a,0})$-action on the first factor of the LHS and the $\tilW\times\pi_0(\calP_a)=\tilW\times\pi_{0}(\hP_{a}/S)$-action on the RHS via the natural map $P_{a,0}\to\hP_{a}$.
\end{prop}
\begin{proof}
The fact that this map intertwines the $\pi_0(P_{a,0})$-action on the LHS and the $\pi_0(\calP_a)$-action on the RHS is clear from the equivariance property of \eqref{preprod}.

It remains to prove the $\tilW$-equivariance of \eqref{embed}. For each standard parahoric $\bP$, we have a commutative diagram
\begin{equation}\label{Cart}
\xymatrix{\SSpr_{a,0}\times Y^{\reg}\ar@{^{(}->}[r]\ar@<-3ex>[d]^{\nu_{\bP}}\ar@<3.5ex>@{=}[d] & \hMpar_{\hata,0}\ar[d]^{\widehat{\forg}^{\bP}_{\bI}}\ar[r]^{\ev_{0}} & [\till_{\bP}/L_{\bP}]\ar[d]^{\pi_{\bP}}\\
\SSpr_{\bP,a,0}\times Y^{\reg}\ar@{^{(}->}[r] & \hM_{\bP,\hata,0}\ar[r]^{\ev_{0}} & [\frl_{\bP}/L_{\bP}]}
\end{equation}
such that the outer square is the product of $Y^{\reg}$ with the diagram \eqref{SprCart} and the right square is the restriction of the diagram \eqref{MCart} at $(\hata,0)\in\calB\times\Xz$. The $W_{\bP}$-action on $\cohoc{*}{\Spr_{a,0}\times Y^{\reg}}$ and $\cohoc{*}{\hMpar_{\hata,0}}$ are constructed from these diagrams using the parahorics $\bP$ and the classical Springer theory for $L_{\bP}$. The diagram \eqref{Cart} implies that \eqref{embed} is $W_{\bP}$-equivariant for every $\bP$, hence $\Wa$-equivariant. Similarly, using a diagram connecting \eqref{Ro} and its global analogue, one shows that \eqref{embed} is also $\Omega_{\bI}$-equivariant. Putting together, we conclude that \eqref{embed} is $\tilW$-equivariant.
\end{proof}

\subsection{Conclusion of the proof of Theorem \ref{th:mainloc}} Consider the diagram
\begin{equation*}
\xymatrix{\cent\otimes\cohoc{i}{\Spr_{a,0}}\ar[r]^{\sigma_{a,0}\otimes\id}\ar@/^2pc/[rr]_{\alpha_{\loc}}\ar@<6ex>[d]^{\eqref{embed}}\ar@{=}@<-6ex>[d] & \Ql[\pi_{0}(P_{a,0})]\otimes\cohoc{i}{\Spr_{a,0}}\ar[r]^(.6){\act_{\loc}}\ar@<6ex>[d]^{\eqref{embed}}\ar@<-5ex>[d]^{\iota_{0}} & \cohoc{i}{\Spr_{a,0}}\ar[d]^{\eqref{embed}}\\
\cent\otimes\cohoc{i+2d}{[\hMpar_{\hata,0}/S]}\ar@/_{2pc}/[rr]^{\alpha_{\glob}}\ar[r]^{\sigma_{a}\otimes\id} & \Ql[\pi_{0}(\calP_{a})]\otimes\cohoc{*}{[\hMpar_{\hata,0}/S]}\ar[r]^(.6){\act_{\glob}} & \cohoc{*}{[\hMpar_{\hata,0}/S]}}
\end{equation*}
where the upper $\alpha_{\loc}$ and $\alpha_{\glob}$ are actions maps of $\cent$; $\act_{\loc}$ and $\act_{\glob}$ are action maps of the $\pi_{0}$'s. By Proposition \ref{p:comppi}, the left side square is commutative. By Proposition \ref{p:comparison}, the right side square and the outer square are commutative. By Corollary \ref{c:mainstalk}, the lower triangle is commutative. Our goal is to prove that the upper triangle is commutative. From the known commutativity, we conclude that $\act_{\loc}\circ(\sigma_{a,0}\otimes\id)$ and $\alpha_{\loc}$ are the same if we further compose them with \eqref{embed}. Since \eqref{embed} is injective, they must be equal before composition, i.e., the upper triangle is commutative. This proves Theorem \ref{th:mainloc} for $\Spr_{a,0}$. Since $a_{0}\equiv a(\gamma)\in\varpi^{N}$ by construction, the theorem also holds for $\Spr_{\gamma}$ by Proposition \ref{p:constancy}. This finishes the proof of Theorem \ref{th:mainloc}.

% homology
 
\subsection{Duality between homology and compactly supported cohomology}\label{ss:duality}
The next goal is to prove Theorem \ref{th:mainhomo}. We consider the following general setting. Let $\Lambda$ be a group. Let $X$ be a scheme, locally of finite type over $k$ with a free $\Lambda$-action such that $X/\Lambda$ is representable by a proper scheme. The main examples we have in mind are $X=\Spr_{\gamma}$ with the action of a lattice $\Lambda\subset G_{\gamma}(F)$.

A {\em $\Lambda$-covering} of $X$ is a morphism
\begin{equation*}
f:Y\times\Lambda\to X
\end{equation*}
such that the induced morphism $\bar{f}:Y\to X/\Lambda$ is finite and surjective. For example, we may take $Y\subset X$ to be the union of representatives of the $\Lambda$-orbits on the irreducible components of $X$, and take $f$ to be the action map.

Given a $\Lambda$-covering, we can form a simplicial resolution $X_{*}$ of $X$
\begin{equation*}
\xymatrix{\cdots & X_{2}\ar@<1ex>[r]\ar[r]\ar@<-1ex>[r] & X_{1}\ar@<.7ex>[r]^{\delta_{0}}\ar@<-.7ex>[r]_{\delta_{1}} & X_{0}\ar[r] & X}
\end{equation*}
where $X_{0}=Y\times\Lambda$ and $X_{n}=X_{0}\times_{X}\times\cdots\times_{X}X_{0}$ ($n+1$ terms). Each $X_{n}$ carries a diagonal $\Lambda$-action. It is easy to see that each $X_{n}$ is $\Lambda$-equivariantly isomorphic to $Y_{n}\times\Lambda$, where $Y_{n}\to X/\Lambda$ is again finite surjective. The morphisms $\delta_{i}:X_{n}\to X_{n-1}$ are all finite.

Using cohomological descent for proper surjective morphisms (see \cite[\S5.3]{HodgeIII}), we see that $\cohoc{*}{X}$ can be calculated as the cohomology of the total complex associated to the following double complex
\begin{equation*}
C^{**}_{c}(f):\cohoc{*}{X_{0}}\xrightarrow{\delta_{0}^{*}-\delta_{1}^{*}}\cohoc{*}{X_{1}}\xrightarrow{\delta_{0}^{*}-\delta_{1}^{*}+\delta_{2}^{*}}\cohoc{*}{X_{2}}\to\cdots
\end{equation*} 
where each term is a free $\Ql[\Lambda]$-module $\cohoc{*}{Y_{n}}\otimes\Ql[\Lambda]$. Dually, $\homog{*}{X}$ can be calculated by the homology of
\begin{equation*}
C_{**}(f):\cdots\to\homog{*}{X_{2}}\xrightarrow{\delta_{0*}-\delta_{1*}+\delta_{2*}}\homog{*}{X_{1}}\xrightarrow{\delta_{0*}-\delta_{1*}}\homog{*}{X_{0}}.
\end{equation*}
where each term is again a free $\Ql[\Lambda]$-module $\homog{*}{Y_{n}}\otimes\Ql[\Lambda]$. In this way, we have upgraded the complexes $\cohoc{*}{X}$ and $\homog{*}{X}$ to objects in the derived category $D^{b}(\Ql[\Lambda]\textup{-mod})$.

Since each $Y_{n}$ is proper, we have
\begin{equation*}
\homog{*}{X_{n}}\cong\homog{*}{Y_{n}}\otimes\Ql[\Lambda]=\Hom_{\Lambda}(\cohoc{*}{Y_{n}}\otimes\Ql[\Lambda],\Ql[\Lambda])\cong\Hom_{\Lambda}(\cohoc{*}{X_{n}},\Ql[\Lambda]).
\end{equation*}
compatible with the differentials in $C^{**}_{c}(f)$ and $C_{**}(f)$, therefore the double complexes $C^{**}_{c}(f)$ and $C_{**}(f)$ are dual as free $\Ql[\Lambda]$-module. This gives an isomorphism in $D^{b}(\Ql[\Lambda]\textup{-mod})$:
\begin{equation}\label{duality}
\homog{*}{X}\cong\bR\Hom_{\Lambda}(\cohoc{*}{X},\Ql[\Lambda]).
\end{equation}

Moreover, neither the isomorphism classes of the objects $\cohoc{*}{X},\homog{*}{X}\in D^{b}(\Ql[\Lambda]\textup{-mod})$ nor the isomorphism \eqref{duality} depend on the choice of the $\Lambda$-covering. In fact, for any two $\Lambda$-coverings $f:Y\times\Lambda\to X$ and $f':Y'\times\Lambda\to X$ are both dominated by the third $f'':Y''\times\Lambda\to X$, where $Y''=(Y\times\{0\})\times_{X}(Y'\times\Lambda)$. Let $g:Y''\to Y$ and $g':Y''\to Y'$ be projections. A further application of cohomological descent gives canonical quasi-isomorphisms
\begin{eqnarray*}
C^{**}_{c}(f)\xrightarrow{g^{*}} C^{**}_{c}(f'')\xleftarrow{g'^{*}} C^{**}_{c}(f'),\\
C_{**}(f)\xleftarrow{g_{*}} C_{**}(f'')\xrightarrow{g'_{*}} C_{**}(f')
\end{eqnarray*}
which intertwine the dualities between $C_{**}(-)$ and $C^{**}_{c}(-)$. Therefore, $f,f'$ and $f''$ all give the same isomorphism \eqref{duality}.

\subsection{Proof of Theorem \ref{th:mainhomo}}\footnote{The idea of proving Theorem \ref{th:mainhomo} by duality of the type \eqref{duality} was suggested by R.Bezrukavnikov.}
Now let $X=\Spr_{\gamma}$ and $\Lambda\subset G_{\gamma}(F)$ be a free abelian subgroup considered in \cite[Proposition 2.1]{KL}, which acts freely on $\Spr_{\gamma}$ with proper quotient \cite[Proposition 3.1(b), Corollary 3.1]{KL}. In fact, the proof in \cite{KL} also applies to $\Spr_{\bP,\gamma}$ for any parahoric $\bP\subset G(F)$. Hence $\Lambda$ also acts freely on $\Spr_{\bP,\gamma}$ with proper quotient.

The discussion in \S\ref{ss:duality} gives the upgraded objects
\begin{equation}\label{up}
\cohoc{*}{\Spr_{\gamma}}, \homog{*}{\Spr_{\gamma}}\in D^{b}(\Ql[\Lambda]\textup{-mod})
\end{equation}
and the canonical isomorphism \eqref{duality} now reads
\begin{equation}\label{dualitySpr}
\homog{*}{\Spr_{\gamma}}\cong\bR\Hom_{\Lambda}(\cohoc{*}{\Spr_{\gamma}},\Ql[\Lambda]).
\end{equation}

\begin{lemma}
The upgraded objects in \eqref{up} carry $\tilW\times\pi_{0}(P_{a(\gamma)})$-actions (lifting the actions on the plain vector spaces), and the isomorphism \eqref{dualitySpr} is $\tilW\times\pi_{0}(P_{a(\gamma)})$-equivariant. 
\end{lemma}
\begin{proof}
For each parahoric $\bP$, we pick a $\Lambda_{\bP}$-covering $f_{\bP}:Y_{\bP}\times\Lambda\to\Spr_{\gamma}$ and define a $\Lambda$-covering $f:Y\times\Lambda\to\Spr_{\gamma}$ by requiring the left square of the following diagram to be Cartesian
\begin{equation}\label{Ycover}
\xymatrix{Y\times\Lambda\ar[r]^{f}\ar[d] & \Spr_{\gamma}\ar[r]\ar[d] & [\till_{\bP}/L_{\bP}]\ar[d]\\
Y_{\bP}\times\Lambda\ar[r]^{f_{\bP}} & \Spr_{\bP,\gamma}\ar[r] & [\frl_{\bP}/L_{\bP}]}
\end{equation}
where the right square is topologically Cartesian by \eqref{SprCart}. The construction of the $W_{\bP}$-action on $\cohoc{*}{\Spr_{\gamma}}$ (resp. $\homog{*}{\Spr_{\gamma}}$) in \S\ref{ss:L} then gives a $W_{\bP}$-action on $C^{**}_{c}(f)$ (resp. $C_{**}(f)$), which is intertwined by the duality between $C^{**}_{c}(f)$ and $C_{**}(f)$ as $\Ql[\Lambda]$-modules. This gives the $W_{\bP}$-action on the upgraded objects \eqref{up}, and proves that \eqref{dualitySpr} is $W_{\bP}$-equivariant. The actions of $\Omega_{\bI}$ and $\pi_{0}(P_{a(\gamma)})$ are given by the actions of $\Omega_{\bI}$ and $P_{a(\gamma)}$ on $\Spr_{\gamma}$ itself, which clearly lift to the objects in \eqref{up} and are intertwined by \eqref{dualitySpr}. One easily checks that the $\pi_{0}(P_{a(\gamma)})$-action commutes with the $W_{\bP}$ and $\Omega_{\bI}$ actions. Using a variant of the diagram \eqref{Ro} incorporating the $\Lambda$-coverings as in \eqref{Ycover}, one checks that the commutation relation between $W_{\bP}$ and $\Omega_{\bI}$ continues to hold after upgrading. This finishes the proof.
\end{proof}

By \eqref{dualitySpr} and the above lemma, we have a $\tilW\times\pi_{0}(P_{a(\gamma)})$-equivariant spectral sequence
\begin{equation*}
E^{-p,-q}_{2}=\Ext^{-p}_{\Lambda}(\cohoc{q}{\Spr_{\gamma}},\Ql[\Lambda])\Rightarrow\homog{p+q}{\Spr_{\gamma}},
\end{equation*}
which necessarily converges because $\Ql[\Lambda]$ has cohomological dimension $\rk(\Lambda)$. Therefore, this gives a finite decreasing filtration $\Fil^{p}$ on $\homo{i}{\Spr_{\gamma}}$ such that
\begin{equation*}
\Gr_{\Fil}^{p}\homog{i}{\Spr_{\gamma}}=E^{p,-i-p}_{\infty}.
\end{equation*}
Since the $\cent$-action on the $E_{2}$ page factors through $\pi_{0}(G_{\gamma}(F))$ by Theorem \ref{th:mainloc}, so does the $\cent$-action on $E_{\infty}$. Therefore, the $\cent$-action on $\Gr^{p}_{\Fil}\homog{i}{\Spr_{\gamma}}$ also factors through $\pi_{0}(G_{\gamma}(F))$. Since $E^{p,q}_{2}=0$ unless $0\leq p\leq \rk(\Lambda)$, the same is true for $E^{p,*}_{\infty}=\Gr^{p}_{\Fil}\homog{*}{\Spr_{\gamma}}$. Note that $\rk(\Lambda)$ is the same as the $F$-rank of $G_{\gamma}$. This proves Theorem \ref{th:mainhomo}.

\appendix

\section{Sheaves and correspondences on spaces locally of finite type}
In this appendix, all algebraic spaces are locally of finite type over $k$.
\subsection{The category of sheaves}\label{a:lft}
Let $X$ be an algebraic space over $k$ which is locally of finite type. Let $\Ft(X)$ be the set of open subsets $U\subset X$ which are of finite type over $k$. We define
\begin{equation*}
\limD^{b}(X):=\varprojlim_{U\in\Ft(X)}D^{b}(U)
\end{equation*}
When $X$ itself is of finite type over $k$, $\Ft(X)$ has a final object $X$, so obviously $\limD^{b}(X)=D^{b}(X)$.

Concretely, an object in $\limD^{b}(X)$ is a system of complexes $\calF_{U}\in D^{b}(U)$ for each open subset $U\subset X$ of finite type over $k$, together with isomorphisms $\varphi^{U}_{V}:j^{*}\calF_{U}\isom\calF_{V}$ for each open embedding $j:V\hookrightarrow U$ satisfying obvious transitivity conditions. A morphism $\alpha:\{\calF_{U}\}\to\{\calG_{U}\}$ is a system of maps $\alpha_{U}:\calF_{U}\to\calG_{U}$ in $D^{b}(U)$ such that $\alpha_{U}$ restricts to $\alpha_{V}$ on $V$.

Examples of objects in $\limD^{b}(X)$ include the constant sheaf $\const{X}:=\{\const{U}\}$ and the dualizing complex $\DD_{X}:=\{\DD_{U}\}$.

\subsection{Functors}
Let $f:X\to Y$ be a morphism which is locally of finite type. We have the following functors
\begin{enumerate}
\item $f^{*}:\limD^{b}(X)\to\limD^{b}(Y)$. For $U\in\Ft(X)$, $f(U)$ is contained in some $V\in\Ft(Y)$. Denote by $f_{U,V}:U\to V$ the restriction of $f$. We define $(f^{*}\calG)_{U}=f_{U,V}^{*}\calG_{V}$.

\item $f^{!}:\limD^{b}(X)\to\limD^{b}(Y)$, defined in a similarly way as $f^{*}$: $(f^{!}\calG)_{U}:=f_{U,V}^{!}\calG_{V}$.

\item If $f$ is of finite type, we have
\begin{equation*}
f_{!}:\limD^{b}(X)\to\limD^{b}(Y)
\end{equation*}
For $V\in\Ft(Y)$, $f^{-1}(V)\in\Ft(X)$. Let $f_{V}:f^{-1}(V)\to V$ be the restriction of $f$. We define $(f_{!}\calF)_{V}:=f_{V,!}\calF_{f^{-1}(V)}$.

In general, if $f$ is only locally of finite type, we have
\begin{equation*}
f_{!}:\limD^{b}(X)\to\ind\limD^{b}(Y)
\end{equation*} 
where $\ind\limD^{b}(Y)$ denotes the category of ind-objects in $\limD^{b}(Y)$. We define $f_{!}\calF$ as the ind-object $\varinjlim_{U\in\Ft(X)}f_{U,!}\calF_{U}$, where $f_{U}:U\to Y$, the restriction of $f$, is of finite type, and $f_{U,!}$ is defined above.

\item If $f$ is of finite type, we have
\begin{equation*}
f_{*}:\limD^{b}(X)\to\limD^{b}(Y)
\end{equation*}
defined in a similar way as $f_{!}$: $(f_{*}\calF)_{V}:=f_{V,*}\calF_{f^{-1}(V)}$.

In general, if $f$ is only locally of finite type, we have
\begin{equation*}
f_{*}:\limD^{b}(X)\to\pro\limD^{b}(Y)
\end{equation*} 
where $\pro\limD^{b}(Y)$ denotes the category of pro-objects in $\limD^{b}(Y)$. We define $f_{*}\calF$ as the pro-object $\varprojlim_{U\in\Ft(X)}f_{U,*}\calF_{U}$.
\end{enumerate}

In particular, we can still define
\begin{equation*}
\homo{*}{X/Y}:=f_{!}\DD_{X}\in\ind\limD^{b}(Y).
\end{equation*}
When $Y=\Spec k$, we have
\begin{eqnarray*}
\cohoc{*}{X}=f_{!}\const{X}, \hspace{.5cm} \homog{*}{X}=f_{!}\DD_{X}\in\ind D^{b}(\Ql\textup{-vector spaces});\\
\cohog{*}{X}=f_{*}\const{X}, \hspace{.5cm} \hBM{*}{X}=f_{*}\DD_{X}\in\pro D^{b}(\Ql\textup{-vector spaces}).
\end{eqnarray*}

\subsection{Cohomological correspondences}\label{ss:corr}
In this appendix, we extend the formalism of cohomological correspondences (see \cite{SGA5} and \cite[Appendix A]{GS}) to situations where the relevant algebraic spaces are locally of finite type.

Consider a correspondence diagram
\begin{equation}\label{d:corr}
\xymatrix{ & C\ar[dl]_{\overleftarrow{c}}\ar[dr]^{\overrightarrow{c}} & \\
 X\ar[r]_{f} & S & Y\ar[l]^{g}}
\end{equation}
where
\begin{itemize}
\item $S$ is locally of finite type over a field $k$;
\item $f,g$ are locally of finite type;
\item $\oright{c}$ is proper and $\oleft{c}$ is of finite type.
\end{itemize}
For $\calF\in D^{b}(X)$ and $\calG\in D^{b}(Y)$, we define as in \cite[Definition A.1.1]{GS}
\begin{equation*}
\Corr(C;\calF,\calG):=\Hom_{\limD^{b}(C)}(\overrightarrow{c}^*\calG,\overleftarrow{c}^!\calF).
\end{equation*}
We call an element $\zeta\in\Corr(C;\calF,\calG)$ a {\em cohomological correspondence} between $\calF$ and $\calG$ with support on $C$.

Given $\zeta\in\Corr(C;\calF,\calG)$, we define
\begin{equation*}
\zeta_{\#}:g_{!}\calG\xrightarrow{g_!(\adj)}g_{!}\oright{c}_{*}\oright{c}^{*}\calG\xrightarrow{g_{!}\oright{c}_{*}\zeta}g_{!}\oright{c}_{*}\oleft{c}^{!}\calF=g_{!}\oright{c}_{!}\oleft{c}^{!}\calF=f_{!}\oleft{c}_{!}\oleft{c}^{!}\calF\xrightarrow{f_{!}(\adj)}f_{!}\calF.
\end{equation*}
In the equality above, we used $\oright{c}_{!}=\oright{c}_{*}$ since it is proper. Arrows indexed by ``$\adj$'' all come from the relevant adjunctions for the morphisms $\oleft{c}$ and $\oright{c}$, which are of finite type. Note that $\zeta_{\#}$ is a morphism in $\ind D^{b}(S)$. 

Most of the results in \cite[Appendix A]{GS} are still valid in this extended situation. In particular, the results on pull-backs of cohomological correspondences in \cite[Appendix A.2]{GS} extends {\em verbatim}.

\subsection{Composition}
Suppose we have the following diagram
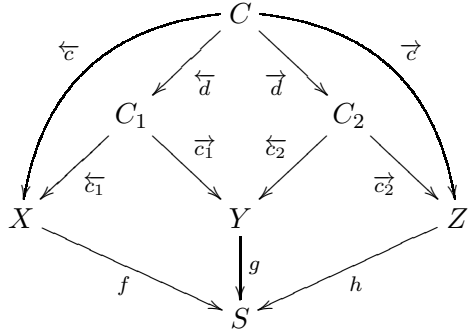
\begin{equation}\label{d:comp}
\xymatrix{& &
C\ar[dl]^{\overleftarrow{d}}\ar@/_2pc/[ddll]_{\overleftarrow{c}}\ar[dr]_{\overrightarrow{d}}\ar@/^2pc/[drdr]^{\overrightarrow{c}} & & \\
& C_1\ar[dl]^{\overleftarrow{c_1}}\ar[dr]^{\overrightarrow{c_1}} & & C_2\ar[dl]_{\overleftarrow{c_2}}\ar[dr]_{\overrightarrow{c_2}} &\\
X\ar[drr]_{f} & & Y\ar[d]^{g} & & Z\ar[dll]^{h}\\
& & S & &}
\end{equation}
where $C=C_{1}\times_{Y} C_{2}$ and $C_{1}$ and $C_{2}$ satisfy the conditions in beginning of \S\ref{ss:corr}. Since $\oright{c_{1}},\oright{c_{2}}$ are proper, so are $\oright{d}$ and $\oright{c}$. Similarly, $\oleft{c}$ is of finite type. Hence $C$, as a correspondence between $X$ and $Z$, also satisfies the conditions in the beginning of \S\ref{ss:corr}.

Let $\calF\in\limD^{b}(X),\calG\in\limD^{b}(Y)$ and $\calH\in\limD^{b}(Z)$. The convolution product defined in \cite[Appendix A.2]{GS} extends to the current situation, giving a bilinear map
\begin{equation*}
\circ:\Corr(C_{1};\calF,\calG)\otimes\Corr(C_{2},\calG,\calH)\to\Corr(C;\calF,\calH).
\end{equation*}
The following statement is a variant of \cite[Lemma A.2.1]{GS}, and is proved by a diagram-chasing:

\begin{lemma}\label{l:comp}
Let $\zeta_{1}\in\Corr(C_{1};\calF,\calG)$ and $\zeta_{2}\in\Corr(C_{2};\calG,\calH)$. Then
\begin{equation*}
(\zeta_{1}\circ\zeta_{2})_{\#}=\zeta_{1,\#}\circ\zeta_{2,\#}:h_{!}\calH\to f_{!}\calF.
\end{equation*}
\end{lemma}
The associativity of the convolution $\circ$ also holds, see \cite[Lemma A.2.2]{GS}.

\subsection{Property (G-2)}\label{a:G2}
From now on we assume both $X$ and $Y$ are smooth of equidimension $d$. Recall from \cite[Appendix A.6]{GS} that we say $C$ has Property (G-2) with respect to an open subset $U\subset S$ if $\dim C_U\leq d$ and the image of $C-C_U\to X\times_{S}Y$ has dimension $<d$.

\cite[Lemma A.6.2]{GS} now reads

\begin{lemma}\label{l:openpart}
Suppose $C$ satisfies (G-2) with respect to $U\subset S$. Let $\zeta,\zeta'\in
\Corr(C;\const{X},\const{Y})$. If $\zeta|_U=\zeta'|_U\in\Corr(C_U;\const{X_U},\const{Y_U})$, then
$\zeta_\#=\zeta'_\#\in\Hom_S(g_!\const{Y},f_!\const{X})$. 
\end{lemma}

\subsection*{Acknowledgement}
The author would like to thank R.Bezrukavnikov, B-C.Ng\^o and Y.Varshavsky for helpful discussions.

\end{document}